\documentclass[11pt]{amsart}
\usepackage[T1]{fontenc}
\usepackage[utf8]{inputenc}
\usepackage[body={15cm, 22cm}]{geometry}
\usepackage{mathrsfs, amssymb, amsrefs}
\usepackage{graphicx,pgf,tikz, enumitem}
\usetikzlibrary{arrows}

\newtheorem{theorem}{\textbf{Theorem}}[section]
\newtheorem{proposition}[theorem]{\textbf{Proposition}}
\newtheorem{lemma}[theorem]{\textbf{Lemma}}
\newtheorem{corollary}[theorem]{\textbf{Corollary}}

\theoremstyle{definition}
\newtheorem{definition}[theorem]{\textbf{Definition}}

\theoremstyle{remark}

\newtheorem{remark}[theorem]{Remark}

\numberwithin{equation}{section}

\allowdisplaybreaks

\def\i\emph{em}

\def\forall{\hbox{for all}~}

\def\v{\vskip 1em}

\def\bega{\begin{array}}\def\enda{\end{array}}
\def\begi{\begin{itemize}}\def\endi{\end{itemize}}

\def\bel{\begin{equation}\label}\def\eeq{\end{equation}}
\def\sqr#1#2{\vbox{\hrule height .#2pt\hbox{\vrule width .#2pt height #1pt \kern #1pt\vrule width .#2pt}\hrule height .#2pt }}

\title{A Debt Management Problem with Currency Devaluation}
\author{Antonio Marigonda}
\address{\hspace{-0.5em}\begin{tabular}{ll}Antonio Marigonda:&Department of Computer Science,\\& University of Verona\\ &Strada Le Grazie 15, I-37134 Verona, Italy.\end{tabular}}
\email{antonio.marigonda@univr.it}

\author{Khai T. Nguyen}
\address{\hspace{-0.5em}\begin{tabular}{ll}Khai T. Nguyen:&Department of Mathematics,\\& North Carolina State University\\ &2108 SAS Hall Box 8205, Raleigh, North Carolina 27695, USA.\end{tabular}}
\email{khai@math.ncsu.edu}

\begin{document}

\begin{abstract}
We consider a model of debt management, where a sovereign state trade some bonds to service the debt with a pool of risk-neutral competitive foreign investors. 
At each time, the government decides which fraction of the gross domestic product (GDP) must be used to repay the debt, and how much to devaluate its currency. 
Both these operations have the effect to reduce the actual size of the debt, but have a social cost in terms of welfare sustainability. 
Moreover, at any time the sovereign state can declare bankruptcy by paying a correspondent bankruptcy cost. 
We show that this optimization problems admits an equilibrium solution, leading to bankruptcy or to a stationary state, depending on the initial conditions.
\end{abstract}

\maketitle

\section{Introduction}
According to US Senate Levin-Coburn Report \cite{LCR}, the financial crisis of 2007-2008, which originated the worldwide Great Recession of 2008--2012 
and to the European sovereign debt crisis of 2010-2012, and whose effects are still present in many countries,  
``was not a natural disaster, but the result of high risk, complex financial products, undisclosed conflicts of interest; 
and the failure of regulators, the credit rating agencies, and the market itself to rein in the excesses of Wall Street.'' 
The first part of the report analyze some topic cases of 
\begin{itemize}
\item [(1)] High Risk Lending;
\item [(2)] Regulatory Failure;
\item [(3)] Inflated Credit Ratings;
\item [(4)] Investment Bank Abuses.
\end{itemize}
In the final recommendation of the report, a whole section is devoted to the management of high risk lending, in order to prevent abuses.\par\medskip\par
In the Eurozone, the crisis - whose consequences lasted until 2016 - took the form of a speculative attack to the sovereign debt of some EU countries (Portugal, Ireland, Greece, Spain), 
but also strongly affects also two major economic powers like Italy and France. The undertaken actions of the EU governments to face the crisis had very high social costs, 
leading also to an heavy political impact. These considerations lead to the following natural problems:
\begin{itemize}
\item to identify suitable tools to estimate the risk of a borrower's bankruptcy (as in the subprime mortgage crisis, which originated);
\item to have quantitative tools, relying on reliable prediction of realistic models, which would allow the regulation authority to prevent abuses;
\item to provide optimal strategies in the management of sovreign debts.
\end{itemize}
In \cite{NT}, the authors introduced a variational model where a government issues nominal defaultable
debt and chooses fiscal and monetary policy under discretion. In particular, to reduce the
actual size of the debt, the government can choose to devaluate its currency, producing inflaction 
and thus increasing the welfare cost and negatively affecting the trust of the investors, or to rely only on fiscal policy to serve the debt.
The government can also declare the default, which imply to pay a bankruptcy cost due to the temporary exclusion
from capital markets, and a drop in the output endowment. 
The aim is to find a strategy minimizing a cost functional dealing with the trade-off between inflaction,
social costs, and debt sustainability and possibly declaring the default if this option would me preferable to
continue servicing the debt.\par\medskip\par
The analysis of the model in \cite{NT} was performed by a numerical methods, and as a final conclusion of their analysis, the authors 
claim that the tool of currency devaluation, though useful in a short-term perspective, is not recommended unless the government
is able to make credible commitments about their future inflaction policy.
In this sense, it is worth of notice that many countries with limited inflaction credibility, decide either to issue bonds \emph{directly} in a foreign
stable currency (e.g., US dollars), or delegates the monetary policy to an independent authority with a strong anti-inflaction commitment (e.g., Eurozone Central Bank).\par\medskip\par
An analytical study of a variant of the model in \cite{NT} was performed in \cite{BMNP}, in the case where no currency devaluation is available to the governement,
and provided a semi-explicit formula for the optimal strategy in the deterministic case (i.e., when the GDP evolves deterministically).
\par\medskip\par
This paper extend the analytical study of \cite{BMNP}, allowing also the possibility of currency devaluation as in \cite{NT}.
Indeed, from the point of view of the model, one of the most relevant resul is if the initial debt-to-income is sufficiently high, \emph{every} optimal debt-management 
strategy needed to employ the currency devaluation.\par\medskip\par 
As a real example of such policy, we mention that, according to the World Economic Outlook (April 2019) of International Monetary Fund, Japan
leads the ranking of countries with higher debt-to-income ratio with a 237.5\%, and it devaluated its currency of about 40\% vs US dollars in the period 2012-2016 (USDJPY index).
\par\medskip\par
The paper is structured as follows: in Section \ref{sec:S-model} we introduce the stochastic model together with the main assumptions, in Section \ref{sec:stoch-analysis} we prove
the existence of an equilibrium solution for the stochastic model as the steady state of an auxiliary parabolic system, 
and study its asymptotic behaviour as the maximum debt-to-income threshold is pushed to $+\infty$.
In Section \ref{sec:det-model} we study the deterministic model obtaining by setting the volatility $\sigma=0$. In this case we provide a semi-explicit construction for an
equilibrium solution, together with a study of its asymptotic behaviour as the maximum debt-to-income threshold is pushed to $+\infty$.

\section{A model with stochastic growth}\label{sec:S-model}

In this section, we develop the model in \cite{BMNP}, allowing the possibility of currency devaluation as in \cite{NT}. 
Here the borrower is a sovereign state, that can decide to devaluate its currency (for example, printing more paper money). 
The total income $Y$, i.e., the gross national product GDP measured in terms of the floating currency unit, 
can quickly increase if the currency is devaluated, producing inflaction. 
It is governed by a stochastic process
\begin{equation*}
dY(t)=(\mu +\tilde v(t)) Y(t)\,dt + \sigma Y(t)\,dW(t),
\end{equation*}
where $W(t)$ is a Brownian motion on a filtered probability space and 
\begin{itemize}
\item $\mu$ is the average growth rate of the economy;

\item $\sigma$ is the volatility;

\item $\tilde v(t)\geq 0$ is the devaluation rate at time $t$, regarded as an additional control.
\end{itemize}
We refer to \cite{NT} for a more detailed derivation $v$ in the above system from economic primitives.
\par\medskip\par
Let $X(t)$ be the outstanding stock of nominal government bonds, expressed in the local currency unit. 
In particular, $X(t)$ represents also the total nominal value of the outstanding debt.
To service the debt, the government trades a nominal non-contingent bond with risk-neutral competitive foreign investors.
In case of bankruptcy, the lenders recover only a fraction $\theta\in[0,1]$ of their outstanding capital, depending on the total amount of the debt at the time of bankruptcy. 
To offset this possible loss, the investors buy a bond with unit nominal value at a discounted price $\tilde p(t)\in [0,1]$.
We denote by $U(t)$ the rate of payments that the borrower chooses to make to the lenders at time $t$. 
If this amount is not enough to cover the running interest and pay back part of the principal, new bonds are issued, at the discounted price $p(t)$. 
As in \cite{BMNP}, the nominal value of the outstanding debt thus evolves according to
\[\dot X(t)=-\lambda X(t) + \dfrac{(\lambda + r) X(t) - U(t)}{\tilde p(t)}.\]
Here the constants 
\begin{itemize}
\item $\lambda$ represents the rate at which the borrower pays back the principal;
\item $r$ represents the discount rate.
\end{itemize}
\par\medskip\par
The debt-to-GDP ratio (DTI) is defined as $x(\cdot) = X(\cdot)/Y(\cdot)$ . 
By It\={o}'s formula \cites{O, Shreve}, the evolution of $x(\cdot)$ is 
\begin{equation}\label{eq:debt-GDP-ev}
dx(t)=\left[\left(\dfrac{\lambda+r}{\tilde p(t)}-\lambda+\sigma^2-\mu-\tilde v(t)\right)x(t)-\dfrac{\tilde u(t)}{\tilde p(t)}\right]\,dt-\sigma\,x(t)\,dW(t),
\end{equation}
where $\tilde u = U/Y\in [0,1]$ is the fraction of the total income allocated to reduce the debt. Throughout the following we will assume that $r>\mu$.
\par\medskip\par
In this model, the borrower has three controls: at each time $t$ he can decide the portion $\tilde u(t)\in [0,1]$  of the total income is allocated to repaying the debt, 
he can decide the devaluation rate $\tilde v(t)\in [0,+\infty[$ and he can also decide the time $T_b(x')$ he is going to declare bankruptcy, in correspondence to a DTI value $x'$, 
paying a bankruptcy cost. Furthermore, we assume that 
\begin{itemize}
\item there exists a threshold $x^*>0$ such that if $x(t)$ reaches $x^*$ then the borrower is forced to declare bankruptcy;
\item the borrower decides to declare bankruptcy as soon as $x(t)$ reaches $x'$, where $x'\in [0,x^*]$ is an additional control parameter, 
chosen by the borrower in order to minimize his expected cost. 
\end{itemize}
We are going to consider \emph{control strategy in feedback form}, i.e., $\tilde u(t)=u(x(t))$, $\tilde v(t)=v(x(t))$, for certain measurable maps $u:[0,x^*]\to [0,1]$ and
$v:[0,x^*]\to [0,+\infty[$.
The total expected cost to the borrower, exponentially discounted in time, for having implemented the control strategy $(u(\cdot),v(\cdot),x')$ is then  
given by
\[J[x_0, x',u(\cdot),v(\cdot)]=E\left[\int_0^{T_b(x')} e^{-rt}\Big[L(u(x(t)))+ c(v(x(t)))\Big]\, dt + e^{-r T_b(x')} B\right]_{x(0)=x_0},\]
where 
\begin{itemize}
\item $T_b(x')$ is the random variable bankruptcy time defined by
\[T_b(x')\doteq\inf\{t>0:\,x(t)= x'\},\]
where $x(\cdot)$ solves \eqref{eq:debt-GDP-ev} with $\tilde u(\cdot)=u(x(\cdot))$, $\tilde v(\cdot)=v(x(\cdot))$ and $x(0)=x_0$.\par\noindent 
If \eqref{eq:debt-GDP-ev} has no solution, we set $J[x_0, x',u(\cdot),v(\cdot)]=+\infty$.
We define also
\begin{equation}\label{eq:bankruptcy-time}T^*_b\doteq \inf\{ t>0:\,x(t)= x^*\}.\end{equation}
\item $B$ is the bankruptcy cost, which summarizes the penalties of temporary exclusion from the capital markets, the bad reputation among the investors, 
and the social costs of the default; 
\item $c(v)\ge 0$ is a social cost resulting by devaluation, i.e., the increasing cost of the welfare and of the imported goods;
\item $L(u)\ge 0$ is the cost for the borrower to implement the control strategy $u(\cdot)$, i.e., adversion toward austerity policies and welfare's budget cuts. 
\end{itemize}

Throughout the paper we will assume the following structural conditions on the functions $L$, $c$. 
More precisely, there exist $\delta_0>0$ and $v_{\max}\in ]0,+\infty]$ such that 
\begin{enumerate}
\item[\textbf{(A1)}] \emph{the implementing cost function $L:[0,1[\to\mathbb R$ is continuous on $[0,1[$, twice continuously differentiable for $u\in ]0,1[$, and satisfies}
\[L(0)=0,\quad  L'(u)>0, \quad L''(u)>\delta_0>0\textrm{ for $u\in]0,1[$, and }\lim_{u\to 1^-} L(u)=+\infty.\]
\item[\textbf{(A2)}] \emph{the social cost $c:[0,v_{\max}[\to\mathbb R$, determined by currency devaluation, is continuous on $[0,v_{\max}[$, 
twice continuously differentiable for $v\in ]0,+\infty[$ and satisfies}
\[c(0)=0,\quad c'(v)>0, \quad c''(v)>\delta_0>0\textrm{ for $v\in]0,v_{\max}[$, and }\lim_{v\to v_{\max}^-} c(v)=+\infty.\]
\end{enumerate}
We extend the definition of the function $L,c$ to lower semicontinuous function defined on the whole of $\mathbb R$, keeping the same names, by setting $L(u)=+\infty$ 
for $u\notin [0,1[$ and $c(v)=+\infty$ for $v\notin [0,v_{\max}[$.
With a slight abuse of notation, we will write $L'(0)$, and $c'(0)$ to denote $\displaystyle\lim_{u\to 0^+}L'(u)$ and $\displaystyle\lim_{v\to 0^+}c'(v)$,
respectively. 

\par\medskip\par

By a Dynamic Programming argument, it is never convenient for the borrower to declare bankruptcy unless he is not forced to do so, i.e.,
unless the threshold $x^*$ is reached.

\begin{lemma}\label{lemma:bestbank} 
If \textbf{(A1)-(A2)} hold then for any admissible control strategy $(u(\cdot),v(\cdot),x')$, 
there exists a control strategy $(\hat u(\cdot),\hat v(\cdot),x^*)$ with smaller cost.
\end{lemma} 
\begin{proof}
By contradiction, assume that the borrower implement any strategy $(u(\cdot),v(\cdot),x')$ with $x'<x^*$. 
We can construct a better strategy simply avoiding to declare bankruptcy at $x'$, and switching off the controls $(u,v)$ after having reached $x'$ 
until the threshold $x^*$ is reached.
Define $(\hat u(t),\hat v(t))=(u(t),v(t))$ for $0\le t\le T_b(x')$ and $(\hat u(t),\hat v(t))=(0,0)$ for $T_b(x')\le t\le T^*_b$.
In this case, recalling that $L(0)=c(0)=0$, we have 
\begin{align*}
J[x_0, x',u(\cdot),v(\cdot)]=&E\left[\int_0^{T_b(x')} e^{-rt}\Big[L(u(x(t)))+ c(v(x(t)))\Big]\, dt + e^{-r T_b(x')} B\right]_{x(0)=x_0}\\
=&E\left[\int_0^{T^*_b} e^{-rt}\Big[L(u(x(t)))+ c(v(x(t)))\Big]\, dt + e^{-r T_b(x')} B\right]_{x(0)=x_0}\\
\ge&E\left[\int_0^{T^*_b} e^{-rt}\Big[L(\hat u(x(t)))+ c(\hat v(x(t)))\Big]\, dt + e^{-r T^*_b} B\right]_{x(0)=x_0}\\
=&J[x_0, x^*,\hat u(\cdot),\hat v(\cdot)].
\end{align*}
\end{proof}

\par\medskip\par

Thus from now on we will always assume that $x'=x^*$, and the goal of the borrower is to minimize
\begin{multline}\label{eq:borrowercost}
J[x_0,x^*,u(\cdot),v(\cdot)]=E\left[\int_0^{T^*_b} e^{-rt}\Big[L(u(x(t)))+ c(v(x(t)))\Big]\, dt + e^{-r T^*_b} B\right]_{x(0)=x_0}.
\end{multline}

\par\medskip\par

To complete the model, we need an equation determining the discounted bond price $\tilde p(\cdot)$ in the evolution equation \eqref{eq:debt-GDP-ev} for $x(\cdot)$ . 
For every $x>0$, denote by $\theta(x)$ the salvage rate, i.e., the fraction of the outstanding capital that can be recovered by lenders, 
if bankruptcy occurs when the debt has size $x^*$. As in \cite{BMNP}, assuming that the investors are risk-neutral, 
the discounted bond price coincides with the expected payoff to a lender purchasing a coupon with unit nominal value
\begin{multline}\label{eq:bondprice}
p(x_0)= E\Big[\int_0^{T^*_b}(r+\lambda)\exp\left\{-\int
_0^t \bigl(\lambda +r+v(x(s))\bigr)\,ds\right\}dt \\
+\exp\left\{-\int^{T^*_b}_{0}(r+\lambda+v(x(t)))dt\right\}\cdot \theta(x^*)\Big]_{x(0)=x_0},
\end{multline}
for every initial debt $x_0\in [0,x^*]$. We then set $\tilde p(t)=p(x(t))$.

\par\medskip\par

After having described the model, we introduce the definition of optimal solution in feedback form.

\begin{definition}[Stochastic optimal feedback solution]\label{def:sto-opt-sol}
\emph{In connection with the above model, we say that a triple of functions 
$(\bar u(\cdot), \bar v(\cdot), \bar p(\cdot))$ provides an optimal solution
to the problem of optimal debt management  \eqref{eq:debt-GDP-ev})--\eqref{eq:bondprice} if
\begin{itemize}
\item[(i)] Given the function $\bar{p}:[0,x^*]\to [0,1]$, for every initial value $x_0\in [0, x^*]$ the feedback control strategy $(\bar u(\cdot),\bar v(\cdot))$ 
with stopping time $T^*_b$ as in \eqref{eq:bankruptcy-time} provides an optimal solution to the  stochastic control problem \eqref{eq:borrowercost} subject 
to dynamics \eqref{eq:debt-GDP-ev} with $\tilde{p}(t)=\bar{p}(x(t))$.
\item[(ii)] Given the feedback control strategy $(\bar u(\cdot),\bar{v}(\cdot))$, for every initial value $x_0\in [0,x^*]$, the discounted price $\bar{p}(x_0)$
satisfies \eqref{eq:bondprice}, where $T^*_b$ is the stopping time \eqref{eq:bankruptcy-time} determined by the dynamics \eqref{eq:debt-GDP-ev} 
with $\tilde{p}(t)=\bar{p}(x(t))$.
\end{itemize}
}
\end{definition}
For a given function $p$, we denote by the \emph{value function} $V:[0,x^*]\to[0,+\infty[$ of the control system \eqref{eq:debt-GDP-ev} 
with $\tilde p(t)=p(x(t))$ and cost given by \eqref{eq:borrowercost} is 
\begin{equation}\label{eq:value}
V(x_0)=\inf J[x_0, x^*,u(\cdot),v(\cdot)],
\end{equation}
Under the assumptions \textbf{(A1)-(A2)}, the Hamiltonian function $H:[0,x^*]\times \mathbb R\times [0,1]\to \mathbb R$ 
associated to the dynamics \eqref{eq:debt-GDP-ev} and to the cost functions $L,c$ in \eqref{eq:borrowercost} 
is defined by
\begin{equation}\label{eq:Hamiltonian}
H(x,\xi,p)\doteq-L^\circ\left(\dfrac{\xi}{p}\right)-c^\circ(x\xi)+ \left(\dfrac{\lambda+r}{p} -\lambda-\mu+\sigma^2\right) x\, \xi,
\end{equation}
where $L^\circ, c^\circ$ are the convex conjugate of $L,c$ respectively (see Appendix \ref{app:convex}).

%
The necessary conditions for optimality imply that the value function $V(\cdot)$ solves the second order implicit ODE
\[r V(x)=H(x,V'(x),p(x))+\dfrac{\sigma^2x^2}{2}V''(x).\]
Recalling Lemma \ref{lemma:subdiffcost} and Lemma \ref{lemma:Hamgrad}, 
as soon as the value function $V$ is determined, the optimal  feedback strategies are 
\begin{equation}\label{eq:optfeedback-u}
u^*(V'(x),p(x)):=
\begin{cases}0,&\textrm{ if }\ \ \ \dfrac{V'(x)}{p(x)}\le L'(0),\\ \\
(L')^{-1}\left(\dfrac{V'(x)}{p(x)}\right),&\textrm{ if }\ \ \ \dfrac{V'(x)}{p(x)}>L'(0),
\end{cases}
\end{equation}
and 
\begin{equation}\label{eq:optfeedback-v}
v^*(x,V'(x))~:=\begin{cases}
0,&\textrm{ if}\ \ \ V'(x)x\le c'(0)\,,\\ \\
(c')^{-1}\left(V'(x)x\right),&\textrm{ if }\ \ \ V'(x)x>c'(0).
\end{cases}
\end{equation}

On the other hand,  if the feedback optimal controls $\bar u(x)$ and $\bar v(x)$  are known, then by using Feynman-Kac formula, 
we obtain the second order nonlinear ODE for the discounted bond price $p(\cdot)$ in \eqref{eq:bondprice}
\begin{multline*}
(r+\lambda+\bar v(x))p(x)-(r+\lambda)\\ =\left[ \left(\dfrac{\lambda+r}{p(x)} -\lambda-\mu+\sigma^2-\bar v(x)\right) x - 
\dfrac{\bar u(x))}{p(x)}\right]\cdot p'(x)+ \dfrac{(\sigma x)^2}{2} p''(x).
\end{multline*}
Concerning boundary conditions, we notice that
\begin{itemize}
\item since bankruptcy is istantaneously declared at $x^*$, the optimal cost $V(x^*)$ starting from the DTI level $x^*$ reduces only to the bankruptcy cost $B$,
while if we start from zero DTI level, $x(t)\equiv 0$ is a solution of  \eqref{eq:debt-GDP-ev} yielding $V(0)=0$ since, for this trajectory bankruptcy never occurs, 
and so $T^*_b=+\infty$ and we take $u(0)=v(0)=0$ thus the cost is zero.
\item with the same argument, at the bankruptcy threshold the cost of a unitary bond $p(x^*)$ must be equal to the salvage rate $\theta(x^*)$ since this 
is the only money that will be repayed to the lender, while at the zero DTI level, using the same trajectory described in the previous case, we have $p(0)=1$.
\end{itemize}
Recalling \eqref{eq:Hamiltonian}-\eqref{eq:optfeedback-v}, and Lemma \ref{lemma:Hamgrad}, 
we are thus led to the system of second order implicit ODEs
\begin{equation}\label{eq:Sode}
\begin{cases}
rV(x)=H(x,V'(x),p(x)) + \dfrac{\sigma^2x^2}{2}\cdot V''(x),\\ \\
(r+\lambda+v(x))p(x)-(r+\lambda)=H_{\xi}(x,V'(x),p(x))\cdot p'(x)+\dfrac{(\sigma x)^2}{2}\cdot p''(x),\\ \\
v(x)=\underset{v\geq 0~~}{\mathrm{argmin}}\{c(v)-vxV'(x)\},
\end{cases}
\end{equation}
with the boundary conditions
\begin{align}\label{eq:boundary}
V(0)=0,&&V(x^*)=B&&\textrm{ and }&&p(0)=1,&&p(x^*)=\theta(x^*).
\end{align}
Therefore, an optimal feedback  solution to the problem of optimal debt management  \eqref{eq:debt-GDP-ev}--\eqref{eq:bondprice} will be obtained
by solving the above system of ODEs for the value function $V(\cdot)$ and for the discounted bond price $p(\cdot)$ for a given bankruptcy threshold $x^*$.

\section{Stochastic optimal feedback solutions}\label{sec:stoch-analysis}
In this section, we prove the existence of an optimal feedback solution to the problem of optimal debt management \eqref{eq:debt-GDP-ev}--\eqref{eq:bondprice} 
for a given bankruptcy threshold $x^*$, and then we will study the asymptotic behavior of the solution as $x^*\to +\infty$.
\subsection{Existence of optimal feedback solutions} Given a bankruptcy threshold $x^*$, we introduce the constant 
\begin{equation}\label{eq:theta-min}
\theta_{\min}\doteq\min\left\{\theta(x^*), \dfrac{r+\lambda}{r+\lambda+v_{\max}}\right\}.
\end{equation}
which will be a lower bound of the discount bond price $p$. 
\par\medskip\par
Our main result of this subsection is the following.
\begin{theorem}\label{thm:Exist-Stoch} 
In addition to \textbf{(A1)}-\textbf{(A2)}, assume that 
\begin{enumerate}
\item[\textbf{\emph{(A3)}}] The volatility satisfies $\sigma>0$, the recover fraction at the bankruptcy threshold satisfies $\theta(x^*)>0$ and the devaluation rate is bounded, i.e., $0<v_{\max}<+\infty$.
\end{enumerate}
Then there exists a constant $M^*>0$ such that the system of second-order ODEs \eqref{eq:Sode} with boundary conditions 
\eqref{eq:boundary} admits a solution \[(V(\cdot),p(\cdot)):[0,x^*]\rightarrow\times [0,B]\times [\theta_{\min},1]\] of class $C^2$. 
Moreover, the function $V(\cdot)$ is monotone increasing and $v(x)=0$ for all $x\in\left[0,\min\left\{\dfrac{c'(0)}{M^*},x^*\right\}\right]$.
\end{theorem}

It is well-known (see Theorem 4.1, p.149, in \cite{FR} or Theorem 11.2.2, p. 141, in \cite{O}) that if $(V(\cdot),p(\cdot))$  is a solution to the boundary value problem
\eqref{eq:Sode}-\eqref{eq:boundary}, then a standard result in the theory of stochastic optimization implies that the feedback control strategy 
$(u^*(V'(\cdot),p(\cdot)),v^*(\cdot,V'(\cdot)))$ given by \eqref{eq:optfeedback-u}-\eqref{eq:optfeedback-v} is optimal for the problem \eqref{eq:borrowercost} with dynamics \eqref{eq:debt-GDP-ev}. 
As a consequence, from Theorem \ref{thm:Exist-Stoch} we deduce 

\begin{corollary} 
Under the same assumptions of Theorem \ref{thm:Exist-Stoch}, the debt management problem \eqref{eq:borrowercost} with dynamics \eqref{eq:debt-GDP-ev} 
admits an optimal control strategy in feedback form. Moreover, there is a threshold level for DTI such that the optimal control strategy does not use 
currency devaluation for values of DTI below that threshold. 
\end{corollary}

The main idea of the proof of Theorem \ref{thm:Exist-Stoch} is to define a family of auxiliary uniformly parabolic evolutive problems, indexed by a parameter $\varepsilon>0$, whose 
steady states will provide an approximate solution of \eqref{eq:Sode}.
For each of such problems, we prove the existence of the steady state,  by constructing a compact, convex and positively invariant 
set of functions $(V,p): [0, x^*]\mapsto [0,B]\times [\theta_{\min},1]$. A topological technique will then yield the existence of a steady state $V_\varepsilon(\cdot)$ for the $\varepsilon$-problem.
After having derived some uniform estimates on the approximate solution $V_\varepsilon(\cdot)$, we are able to pass to the limit as $\varepsilon\to 0^+$, obtaining a solution of \eqref{eq:Sode}.
The main difficulty is the degeneration of the coefficient of the second derivative of $V(\cdot)$ in \eqref{eq:Sode}, which must be tackled by a suitable approximation argument.

\begin{proof}[Proof of Theorem \ref{thm:Exist-Stoch}.]
From assumptions \textbf{(A1)-(A3)} and Lemma \ref{lemma:Hamgrad}, for all $(x,\xi)\in [0,+\infty[\times \mathbb R$ it holds 
\begin{align}\label{eq:estoptimal-v}
v^*(x,\xi)\le v_{\max},&&|v^*_{x}(x,\xi)|\le \dfrac{1}{\delta_0}\cdot |\xi|&&\mathrm{ and }&&|v^*_{\xi}(x,\xi)|\le \dfrac{1}{\delta_0}\cdot |x|.
\end{align}
For any $\varepsilon>0$, consider the parabolic system in the unknown $\mathbf{V}=\mathbf{V}(t,x)$, $\mathbf{p}=\mathbf{p}(t,x)$
\begin{equation}\label{eq:parabolic-uniform}
\begin{cases}
\mathbf{V}_t &=~-r\mathbf{V}+ H(x,\mathbf{V}_x,\mathbf{p}+\varepsilon)+ \left(\varepsilon+\dfrac{(\sigma x)^2}{2}\right)\mathbf{V}_{xx},\\ \\
\mathbf{p}_t &=~(r+\lambda)-(r+\lambda+v^*(x,\mathbf{V}_x))\mathbf{p}+ H_{\xi}(x,\mathbf{V}_x,\mathbf{p}+\varepsilon)\mathbf{p}_x+\left(\varepsilon+\dfrac{(\sigma x)^2}{2}\right)\mathbf{p}_{xx},
\end{cases}
\end{equation}
together with the boundary conditions 
\begin{align}\label{eq:parabolic-bd}
\begin{cases}\mathbf{V}(t,0)&=~0,\\ \\ \mathbf{V}(t,x^*)&=~B,\end{cases}&&\begin{cases}\mathbf{p}(t,0)&=~1,\\ \\ \mathbf{p}(t,x^*)&=~\theta(x^*),\end{cases}
\end{align}
and notice that the system is uniformly parabolic also in a neighborhood of $x=0$. The remaining part of the proof is divided into three main steps.
\par\medskip\par
\noindent \textbf{Step 1.} \emph{The system \eqref{eq:parabolic-uniform}-\eqref{eq:parabolic-bd} admits a steady state solutions $(V_{\varepsilon}(\cdot),p_{\varepsilon}(\cdot))$ for every given $\varepsilon>0$. Moreover, for $\varepsilon$ sufficiently small, $V_\varepsilon(\cdot)$ is monotone increasing.}
\par\bigskip\par
\begin{enumerate}
\item [(i).] Recalling Theorem 1 in \cite{A}, for every pair  $(V_0,p_0)\in C^2([0,x^*])\times  C^2([0,x^*])$, the system \eqref{eq:parabolic-uniform}-\eqref{eq:parabolic-bd} with the initial data
\begin{align}\label{eq:parabolic-initial-data}
\mathbf{V}(0,x) = V_0(x)\qquad\textrm{ and }\qquad \mathbf{p}(0,x) = p_0(x)
\end{align}
admits a unique solution $\left(\mathbf{V}(t,x),\mathbf{p}(t,x)\right)$ in $ C^2([0,T]\times [0,x^*])\times  C^2([0,T]\times [0,x^*])$. 
Adopting a semigroup notation, let $t\mapsto S_t(V_0, p_0):=(\mathbf{V}(t,\cdot), \mathbf{p}(t,\cdot))$ be the solution of \eqref{eq:parabolic-uniform}-\eqref{eq:parabolic-bd} 
with initial data \eqref{eq:parabolic-initial-data}. Recalling the definition of $\theta_{\min}$ in \eqref{eq:theta-min},
%
we claim that the following closed, convex set of $ C^2$ functions 
\[
\mathcal{D} = \left\{ (V_0,p_0)\in C^2([0,x^*];[0,B])\times C^2([0,x^*];[\theta_{\min},1]):\, V_0,p_0\textrm{ satisfy \eqref{eq:boundary}}\right\}.
\]
is positively invariant under the semigroup $S_{t}$, namely
\[S_t(\mathcal{D})~\subseteq~\mathcal{D},\textrm{ for all }t\geq 0.\]
Indeed, consider  the constant functions
\begin{align*}\begin{cases}\mathbf{V}^+(t,x)&=~B,\\ \\ \mathbf{V}^-(t,x)&=~0,\end{cases}&&\begin{cases}\mathbf{p}^+(t,x)&=~1,\\ \\  \mathbf{p}^-(t,x)&=~\theta_{\min}.\end{cases}\end{align*} 
Recalling Lemma \ref{lemma:Hprop} (2), we have 
\[H(x,\mathbf{V}^+_{x},\mathbf{p}+\varepsilon) = H(x,\mathbf{V}^-_{x},\mathbf{p}+\varepsilon) = H(x,0,\mathbf{p}+\varepsilon)=0.\]
It can be easily checked that $\mathbf{V}^+$ is a supersolution and $\mathbf{V}^-$ is a subsolution of the first scalar 
parabolic equation in \eqref{eq:parabolic-uniform}. 
Moreover, $\mathbf{p}^+$ is a supersolution and $\mathbf{p}^-$ is a subsolution of the second scalar parabolic equation in \eqref{eq:parabolic-uniform}. 
Therefore, for any pair of functions $(V_0,p_0)$ in \eqref{eq:parabolic-initial-data} taking values in the box $[0, B]\times [\theta_{\min},\, 1]$, 
the solution of the system
\eqref{eq:parabolic-uniform} will satisfy 
\[
0\le \textbf{V}(t,x)\le  B\qquad\textrm{ and }\qquad \theta_{\min}\le \mathbf{p}(t,x)\le 1
\]
for all $(t,x)\in [0.+\infty)\times [0,x^*]$.
\par\medskip\par
\item [(ii).] Thanks to the bounds of Lemma \ref{lemma:Hprop} (1) and \eqref{eq:estoptimal-v} we can apply Theorem 3 in \cite{A} and obtain the existence of a steady state 
$( V_\varepsilon(\cdot), p_\varepsilon(\cdot))\in \mathcal{D}$ for the system 
\eqref{eq:parabolic-uniform}-\eqref{eq:parabolic-bd}, i.e., a solution of
\begin{equation}\label{eq:steady}
\begin{cases}\displaystyle
-rV+ H(x,V',p+\varepsilon) + \left(\varepsilon+\dfrac{(\sigma x)^2}{2}\right)
V'' = 0\,,\\[4mm]
(r+\lambda)-(r+\lambda+v(x))p+ \displaystyle H_{\xi}(x,V',p+\varepsilon) p'
+ \left(\varepsilon+\dfrac{(\sigma x)^2}{2}\right) p'' = 0\,,\end{cases}\end{equation}
with
\begin{equation*}\label{eq:steady-v}
v(x)\doteq v^*(x,V'(x)) = \begin{cases} 0,&\textrm{ if }~~V'(x)x\le  c'(0)\\ \\ (c')^{-1}\left(V'(x)x\right),&\textrm{ if }~~V'(x)x>c'(0),\end{cases}
\end{equation*}
To complete this step, we show that $V_{\varepsilon}$ is monotone increasing for $\varepsilon>0$ sufficiently small. 
Since $V_{\varepsilon}(0)=0$, $V_{\varepsilon}(x^*)=B$ and $V_{\varepsilon}(x)\in [0,B]$ for all $x\in [0,B]$, it holds
\begin{align*}\lim_{x\to 0+}V'_{\varepsilon}(x)\ge  0\qquad\textrm{ and }\qquad\lim_{x\to x^{*}-} V'_{\varepsilon}(x)\ge  0.\end{align*} 
Assume by a contradiction that $V_{\varepsilon}$ is not monotone increasing. Then there exists $x_0\in ]0,x^*[$ such that 
\begin{align*}V'_{\varepsilon}(x_0)<0\qquad\textrm{ and }\qquad V''_{\varepsilon}(x_0)=0.\end{align*}
In particular, the first equation of \eqref{eq:steady} implies that
\[H(x_0,V'_{\varepsilon}(x_0),p_{\varepsilon}(x_0)+\varepsilon) = rV_{\varepsilon}(x_0)\ge  0.\]
On the other hand, for any given $0<\varepsilon<\dfrac{r-\mu}{\lambda+\mu}$, since $p_{\varepsilon}(x_0)\leq 1$, it holds
\[
\dfrac{\lambda+r}{p_{\varepsilon}(x_0)+\varepsilon} -(\lambda+\mu) + \sigma^2\ge \dfrac{\lambda+r}{1+\varepsilon}-(\lambda+\mu) + \sigma^2 = 
~\dfrac{{r-\mu}-(\lambda+\mu)\varepsilon}{1+\varepsilon}+\sigma^2>0.
\]
Thus, Lemma \ref{lemma:Hprop} (1) yields
\[
H(x_0,V'_{\varepsilon}(x_0),p_{\varepsilon}(x_0)+\varepsilon)\le  \left(\dfrac{\lambda+r}{p_{\varepsilon}(x_0)+\varepsilon} -\lambda-\mu + \sigma^2\right) x_0\, V'_{\varepsilon}(x_0)<0,
\]
and it yields a contradiction, since the left hand side is nonnegative.
\end{enumerate}
\par\medskip\par
\noindent \textbf{Step 2.} \emph{We now derive a priori estimate on these stationary solutions, which will allow to pass to the limit as $\varepsilon\to 0^+$.}
\par\bigskip\par
\begin{enumerate}
\item [(i).] We start providing an upper bound for $\|V'_{\varepsilon}\|_{{L}^{\infty}(]0,x^*[)}$. 
For any $(x,\xi,p)\in [0,x^*]\times[0,+\infty[\times]0,1]$, since $c^\circ(v)\ge 0$, $\lambda,\mu\ge 0$, it holds
\begin{align*}
H(x,\xi,p+\varepsilon)~\leq&~\min_{u\in [0,1]}~\left\{ L(u) - \dfrac{u}{p+\varepsilon}\, \xi\right\}+ \left(\dfrac{\lambda+r}{p+\varepsilon}+ \sigma^2\right) x \xi\\
~\leq&~\dfrac{{2(\lambda+r+\sigma^2)x-1}}{2p+2\varepsilon}\cdot \xi+L\left(\dfrac{1}{2}\right).
\end{align*}
In particular, for $0<\varepsilon<1$ , set $\bar{x}_1\doteq\min\left\{\dfrac{1}{4(\lambda+r+\sigma^2)}, \dfrac{x^*}{2}\right\}$ and  $\overline{M}_1\doteq 8 L\left(\dfrac{1}{2}\right)>0$, it holds 
\begin{equation}\label{eq:lo-H}
H(x,\xi,p+\varepsilon)\le  -\dfrac{1}{4 p+4\varepsilon}\cdot \xi+L\left(\dfrac{1}{2}\right)<0
\end{equation}
for all $(x,\xi,p)\in [0,\bar{x}_1]\times[\overline{M}_1,+\infty[\times[0,1]$. As a consequence if $V'_{\varepsilon}(\tilde x)>\overline{M}_1$ for a certain $\tilde x\in [0,\bar x_1]$, from the first equation of \eqref{eq:steady}, 
we have that $V''_{\varepsilon}(\tilde x)>0$, and so $\tilde x$ cannot be a local maximum for $V'_\varepsilon(\cdot)$. 
This implies that either $V_\varepsilon'(x)\le \overline{M}_1$ for all $x\in [0,x^*]$, or $V_\varepsilon'(\cdot)$ attains its maximum in $[\bar{x}_1,x^*]$.\par
Hence, we only need to show that $V'_{\varepsilon}$ is bounded in $[\bar{x}_1,x^*]$. Since $V'_\varepsilon(x)\ge 0$ and $V_\varepsilon(\cdot)$ is bounded from above by $B$, from the first equation of \eqref{eq:steady} and Lemma \ref{lemma:Hprop} (1), one gets
\begin{align*}
|V''_{\varepsilon}(x)|=&\dfrac{2}{2\varepsilon+\sigma^2 x^2}\cdot \left(|H(x,V'_{\varepsilon}(x),p_{\varepsilon}(x))|+rV_{\varepsilon}(x)\right)\\
\leq&\dfrac{2}{\sigma^2\bar{x}_1^2}\cdot \left[\dfrac{1+(r+\lambda)x^*}{\theta_{\min}}+(\sigma^2+\mu+\lambda+v_{\max})\cdot x^*\right]\cdot V'_{\varepsilon}(x)+\dfrac{2rB}{\sigma^2\bar{x}^2_1}
\end{align*}
for all $x\in [\bar{x}_1,x^*]$. On the other hand, by the mean value theorem, there exists a point $\hat x\in [\bar{x}_1, x^*]$ where 
\[V_{\varepsilon}'(\hat x) = \dfrac{V_{\varepsilon}(x^*)- V_{\varepsilon}(\bar{x}_1)}{x^*-\bar{x}_1}\le \dfrac{2B}{x^*}\]
By Gr\"onwall's lemma, from the above differential inequality and the estimate on $V_{\varepsilon}(\hat x)$, we obtain an upper bound on $V_{\varepsilon}'(x)$, uniform in $\varepsilon$, 
for all $x\in [\bar{x}_1, \hat x]\cup[\hat x, x^*]$.
Therefore there exists $M^*>0$ which does not depend on $\varepsilon$ such that 
\begin{equation}\label{eq:upperboundV'}
0\le V'_{\varepsilon}(x)\le  M^*\qquad\forall x\in ]0,x^*[.
\end{equation}
As a consequence, \eqref{eq:steady-v} implies that 
\begin{align}\label{eq:nodev}
0\le v_{\varepsilon}(x)\doteq v^*(x,V_{\varepsilon}'(x))\le  (c')^{-1}\left(M^*x\right)\qquad\forall x\in [0,x^*]
\end{align}
and, in particular, we have the following estimate concerning the use of no devaluation in the control strategy
\begin{align*}
v_{\varepsilon}(x) = 0~~\textrm{ for all } x\in \left[0,\min\left\{\dfrac{c'(0)}{M^*},x^*\right\}\right].
\end{align*}
\item [(ii).] We now provide uniform bounds on $\|V''_{\varepsilon}\|_{L^{\infty}([\delta,x^*[)}$, $\|p'_{\varepsilon}\|_{L^{\infty}([\delta,x^*[)}$ 
and $\|p''_{\varepsilon}\|_{L^{\infty}([\delta,x^*[)}$ for any fixed $0<\delta<x^*$. From Lemma \ref{lemma:Hprop} (1) and \eqref{eq:upperboundV'}, 
recalling that $p_{\varepsilon}(x)\ge\theta_{\min}$, there exist $K_1,K_2>0$, which are independent on $\varepsilon$, such that 
\begin{align*}
\left |H(x,V'_{\varepsilon}(x),p_{\varepsilon}(x))\right|\le  K_1\qquad \textrm{and}\qquad\left |H_{\xi}(x,V'_{\varepsilon}(x),p_{\varepsilon}(x))\right|\le ~K_2~
\end{align*}
for all $x\in [0,x^*]$. Thus, the first equation in \eqref{eq:steady} yields
\begin{equation}\label{eq:boundV''}
\|V''_{\varepsilon}\|_{{L}^{\infty}([\delta,x^*[)}\le \tilde{C}_{1,\delta}\doteq \dfrac{2rB+2K_1}{\sigma^2\delta^2}.
\end{equation}
Moreover, multiplying the second equation in \eqref{eq:steady} by $p'_{\varepsilon}(x)$ and using the estimates on $H_\xi$, we have 
\begin{equation*}
p''_{\varepsilon}(x)p'_{\varepsilon}(x)\le \dfrac{2K_2}{\sigma^2\delta^2}\cdot\big|p'_{\varepsilon}(x)\big|^2+\dfrac{2(r+\lambda+v_{\max})}{\sigma^2\delta^2}\cdot |p'_{\varepsilon}(x)|\quad\forall x\in [\delta,x^*).
\end{equation*}
Set $z_\varepsilon(x)=\dfrac 12\big|p'_{\varepsilon}(x)\big|^2$, we have for all $x\in [\delta,x^*[$
\begin{align*}
\dot z_{\varepsilon}(x)~\leq&~\dfrac{4K_2}{\sigma^2\delta^2}\cdot z_\varepsilon(x)+\dfrac{2\sqrt 2(r+\lambda+v_{\max})}{\sigma^2\delta^2}\cdot\sqrt{z_{\varepsilon}(x)}\\
~\leq&~\dfrac{4K_2}{\sigma^2\delta^2}\cdot z_\varepsilon(x)+\dfrac{\sqrt 2(r+\lambda+v_{\max})}{\sigma^2\delta^2}\cdot(z_{\varepsilon}(x)-1)+1.\nonumber
\end{align*}
By the mean value theorem,  there exists a point $\hat{x}\in [\delta,x^*]$ such that 
\[
z_\varepsilon(\hat x) = |p_{\varepsilon}'(\hat{x})|^2 = \dfrac 12\Big|\dfrac{p_{\varepsilon}(x^*)-p_{\varepsilon}(\delta)}{ x^*-\delta}\Big|^2\le  \dfrac{1}{2(x^*-\delta)^2}.
\]
The Gr\"onwall's lemma applied to $z_{\varepsilon}(\cdot)$, together with the second equation in \eqref{eq:steady}, yield
\begin{align}\label{eq:bound-p''} 
\|p_{\varepsilon}'\|_{L^{\infty}([\delta,x^*[)}\le \tilde{C}_{2,\delta}\qquad \textrm{ and }\qquad \|p_{\varepsilon}''\|_{L^{\infty}([\delta,x^*[)}\le \tilde{C}_{3,\delta}
\end{align}
for some constants $\tilde{C}_{2,\delta},\tilde{C}_{3,\delta}$, and the estimate is valid uniformly as $\varepsilon\to 0^+$.
\end{enumerate}
\v
\noindent \textbf{Step 3.} \emph{We finally prove the existence of a solution to \eqref{eq:Sode}-\eqref{eq:boundary} by letting $\varepsilon\to0^+$.}
\par\bigskip\par
\begin{enumerate}
\item [(i).] Recalling \eqref{eq:Hamiltonian}-\eqref{eq:optfeedback-u}, Lemma \ref{lemma:Hamgrad}, and \eqref{eq:upperboundV'}, we obtain that  $H$ and $H_{\xi}$ are uniformly Lipschitz on $[\delta,x^*]\times [0,M^*]\times[\theta_{\min},1]$. Hence,  the functions 
\[\left(V_{\varepsilon}\right)'' = \dfrac{2}{2\varepsilon+\sigma^2x^2}\cdot \left[rV_{\varepsilon}-H(x,(V_{\varepsilon})',p_{\varepsilon}+\varepsilon)\right]\]
and
\[\left(p_{\varepsilon}\right)'' = \dfrac{2}{2\varepsilon+\sigma^2x^2}\cdot \left[(r+\lambda)\cdot (p_{\varepsilon}-1)-H_{\xi}(x,(V_{\varepsilon})',p_{\varepsilon}+\varepsilon) p'_{\varepsilon}\right]\]
are also uniformly bounded and uniformly Lipschitz on $[\delta, x^*]$. Thus, by Ascoli-Arzel\`a theorem, choosing a suitable sequence $\varepsilon_n\to 0^+$, we have the uniform convergence  
$(V_{\varepsilon_n}, p_{\varepsilon_n})\to (V,p)$ in $C^2(]\delta,x^*[)$ for all $\delta>0$, where $V,p$ are twice continuously differentiable
and solve the system of ODEs \eqref{eq:Sode} on the open interval $]0, x^*[$. Moreover, recalling \eqref{eq:upperboundV'}, \eqref{eq:boundV''}, and \eqref{eq:bound-p''}, it holds
\begin{align*}
\lim_{x\to x^{*-}}V(x) = B,\qquad \lim_{x\to x^{*-}}p(x) = \theta(x^*)\qquad \textrm{and}\qquad\lim_{x\to 0^+}V(x)=0.
\end{align*}
\item [(ii).] To complete the proof, we will show that $\displaystyle\lim_{x\to0^+} p(x)=1$ by providing a lower bound for $p_{\varepsilon}(\cdot)$ in a right neighborhood of $x=0$, 
independent of $\varepsilon$. Consider the function 
\[p^-(x) = 1- c x^\gamma\]
where 
\[
c = 1+\dfrac{2(M^*)^2}{(r+\lambda)\cdot \delta}+\dfrac{1}{x^*}\quad\textrm{and}\quad \gamma = \min\left\{\dfrac{1}{2}\,, ~\dfrac{(r+\lambda)}{\left(\frac{\lambda+r}{ \theta_{\min}} -\lambda-\mu+\sigma^2\right)}\right\}.
\]
We prove that  $p^-(\cdot)$ is a lower solution of the second equation of \eqref{eq:steady} in the interval $[0,\bar{x}_0]$ with 
\[0<x_0 = \left(\dfrac{1-\theta_{\min}}{c}\right)^\frac{1}{\gamma}\le  \min\{1,x^*\}.\]
Indeed, it is clear that 
\begin{equation}\label{eq:boundarycondssol}
p^-(\bar{x}_0) = \theta_{\min}\le p_{\varepsilon}(\bar{x}_0),\quad 1 = p^-(0) = p_{\varepsilon}(0)\quad\textrm{ and }\quad p^{-}(x)\ge  0.
\end{equation}
On the other hand, recalling \eqref{eq:nodev} and assumption \textbf{(A2)}, we have  
\begin{equation*}
v_{\varepsilon}(x) = v^*(x,V_{\varepsilon}'(x))\le  (c')^{-1}(M^*x)~\le ~\frac{(M^*)^2}{\delta}\cdot x\qquad\forall x\in [0,x^*].
\end{equation*}
Thus, by \eqref{eq:steady} and Lemma \ref{lemma:Hamgrad}, we obtain that 
\begin{align*}
(r+\lambda)&-(r+\lambda+v_{\varepsilon}(x))p^-+ \displaystyle H_{\xi}(x,V_{\varepsilon}',p^-) p^-{}'+ \left(\varepsilon+\dfrac{(\sigma x)^2}{2}\right) p^-{}''\\
=&-v_{\varepsilon}(x)+(r+\lambda) c x^\gamma- H_\xi(x, V_{\varepsilon}', p^-)\,c \gamma x^{\gamma-1}+c\left(\varepsilon+\dfrac{(\sigma x)^2}{2}\right)\gamma (1-\gamma)x^{\gamma-2}\\
\geq&~(r+\lambda) c x^\gamma- H_\xi(x, V', p^-)\,c \gamma x^{\gamma-1}-\frac{(M^*)^2}{\delta}\cdot x\\
\geq&~\frac{r+\lambda}{2}\cdot cx^{\gamma}-\frac{(M^*)^2}{\delta}\cdot x\ge \frac{(r+\lambda)x^{\gamma}}{2}\cdot\left(c-\frac{2(M^*)^2}{(r+\lambda)\cdot \delta}\cdot \bar{x}_0^{1-\gamma}\right)\ge 0
\end{align*}
for all $x\in ]0,\bar{x}_{0}]$. Recalling \eqref{eq:boundarycondssol}, a standard comparison argument yields
\[p_{\varepsilon}(x)\ge p^-(x) = 1-cx^{\gamma}\qquad\forall x\in [0,\bar{x}_0].\]
and this implies  
\[p(x) = \lim_{\varepsilon_n\to 0+}~p_{\varepsilon_n}(x)\ge  1-cx^{\gamma}\qquad\forall x\in [0,\bar{x}_0].\]
Since $p(x)\in [0,1]$ for all $x\in[0,x^*]$, we conclude that  $\displaystyle\lim_{x\to 0^+}p(x)=1$.
\end{enumerate}
\end{proof}

\subsection{Dependence on the bankruptcy threshold $x^*$} 
In this subsection, we will study the behavior of the expected total optimal cost for servicing the debt when the maximum size $x^*$ of the DTI, at which bankruptcy is declared, 
becomes very large.  More precisely, for a given $x^*>0$,  let $(V(\cdot,x^*),p(\cdot, x^*))$ be a solution to the system of second order ODEs \eqref{eq:Sode} with boundary conditions 
\eqref{eq:parabolic-bd}. We investigate whether, as $x^*\to\infty$, the value function $V(\cdot,x^*)$ remains strictly positive or approaches zero uniformly on bounded sets.
\par\medskip\par
From the point of view of the model, the latter situation corresponds to a \emph{Ponzi's scheme}, where no actual effort is made by the borrower to repay the lenders,
and the debt and its interests are served just by borrowing more money from the investors.
\par\medskip\par
It turns out that a crucial role in the asymptotic behavior of $V(\cdot,x^*)$ as $x^*\to +\infty$ is played by the \emph{speed of decay} of the salvage rate $\theta(x^*)$
as $x^*\to +\infty$, which represents the fraction of the investment that can be recovered by the investors after the bankruptcy (and the unitary bond discounted price at the bankruptcy threshold). 
\par\medskip\par
If the salvage rate decay sufficiently slowly, i.e., the lenders can still recover a sufficiently high fraction of their investment after the bankruptcy, then the best choice for the borrower is to
implement the Ponzi's scheme. On the other hand, if the salvage rate $\theta(x^*)$ decays sufficiently fast, then Ponzi's scheme is no longer an optimal solution for the borrower.

\begin{theorem} 
Under the same assumptions as in Theorem \ref{thm:Exist-Stoch}, the followings hold:
\begin{itemize}
\item [(i)] if $\displaystyle\liminf_{s\to+\infty}\theta(s)s=+\infty$ then 
\begin{equation}\label{eq:stoch-ponzi}
\lim_{x^*\to\infty}V(x,x^*) = 0\qquad\forall x\in [0,\infty[
\end{equation}
\item [(ii)] if $\displaystyle\limsup_{s\to+\infty}\theta(s)s=C_1<+\infty$ then 
\begin{equation}\label{eq:stoch-nonponzi}
\liminf_{x^*\to\infty}V(x,x^*)\ge \dfrac{rB}{2(2r+v_{\max}-\mu)}
\end{equation}
for all $x\geq \dfrac{2\Big(r+v_{\max}+(1+C_1(r+v_{\max}))(r-\mu)\Big)}{r(r-\mu)}$.
\end{itemize}
\end{theorem}
\par\medskip\par
\begin{proof}
For any fixed $x^*>0$, let $(V(\cdot,x^*),p(\cdot,x^*))$ be a solution to \eqref{eq:Sode} with boundary conditions \eqref{eq:parabolic-bd} and set 
\[
v(x,x^*)\doteq \displaystyle\underset{\omega\geq 0}{\mathrm{argmin}}\{c(\omega)-\omega xV'(x,x^*)\}.
\]
With the same argument of the proof of Theorem \ref{thm:Exist-Stoch}, we obtain that 
\[(V(x,x^*),p(x,x^*))\in [0,B]\times[\theta_{\min},\theta(x^*)],\] 
and $V(\cdot,x^*)$ is increasing. Moreover, there exists $M^*>0$  depending on $x^*$ such that $V'(x,x^*)\leq M^*$ for all $x\in [0,x^*]$ and 
\begin{equation}\label{eq:nondev-asy}
v(x,x^*) = 0\qquad\forall x\in \left[0,\bar{x}_0\doteq \dfrac{c'(0)}{M^*}\right]\,.
\end{equation}
In order to achieve $(i)$ and $(ii)$, we construct for upper and lower bounds for $(V(\cdot, x^*)$, $p(\cdot, x^*))$, in the following form
\begin{align*}
V_2(x)\leq V(x, x^*)\le V_1(x)\textrm{ and }p_1(x)\le p(x, x^*)\le p_2(x)
\end{align*}
where
\begin{itemize}
\item 
the function $V_1(\cdot)$ and $V_2(\cdot)$ are a supersolution and a subsolution of the first equation in \eqref{eq:Sode}, respectively.  
\item 
the functions $p_1(\cdot)$ and $p_2(\cdot)$ are a subsolution and a supersolution of the second equation in \eqref{eq:Sode}, respectively.    
\end{itemize}
To this aim, we introduce two constants
\begin{align*}
\gamma\doteq \dfrac{r+\lambda}{r+\lambda+v_{\max}}\qquad\textrm{ and }\qquad\beta\doteq  r+\lambda+v_{\max}.
\end{align*}
\par\medskip\par
\noindent \textbf{Step 1.} We first prove $(i)$. Suppose that 
\begin{equation}\label{eq:asp0}
\liminf_{s\to+\infty}\theta(s)s = +\infty\,.
\end{equation}
We shall construct a suitable pair of functions $(V_1(\cdot),p_1(\cdot))$. Two cases are considered:
\par\medskip\par
\begin{itemize}
\item If $\theta(x^*)\geq \gamma$ then let $V_1(\cdot)$  be the solution to the backward Cauchy problem 
\begin{align*}
rV_1(x)=\left(\beta+\sigma^2\right)xV_1'(x)\qquad\mathrm{with}\qquad V_1(x^*)=B.
\end{align*}
Solving the above ODE, we obtain that 
\begin{align*}
V_1(x)=B\cdot \left(\dfrac{x}{x^*}\right)^{\frac{r}{\beta+\sigma^2}}\qquad\textrm{ for all }x\in [0,x^*]. 
\end{align*}
Since $0<\dfrac{r}{\beta+\sigma^2}<1$, it holds that  $V_1'(x)>0$ and $V_1''(x)<0$ for all $x\in ]0,x^*[.$
Thus, from Lemma \ref{lemma:Hprop} (1), for all $q\geq \gamma$ it holds 
\begin{align*}
-rV_1(x)+H(x,V_1'(x),q)+\dfrac{\sigma^2x^2}{2}\cdot V_1''(x)~\leq&~-rV_1(x)+\left(\dfrac{r+\lambda}{\gamma}+\sigma^2\right)xV_1'(x)\\
=&~-rV_1(x)+\left(\beta+\sigma^2\right)xV_1'(x)=0\,.
\end{align*}
In particular, since from Theorem \ref{thm:Exist-Stoch} and \eqref{eq:theta-min}, it holds 
\[p(x,x^*)\ge \theta_{\min} = \min\left\{\theta(x^*),\dfrac{r+\lambda}{r+\lambda+v_{\max}}\right\} = \gamma,\]
we have
\[
-rV_1(x)+H(x,V_1',p(x,x^*))+\dfrac{\sigma^2x^2}{2}\cdot V_1''(x)\le 0\,.
\]
Thus $V_1$ is a super-solution of the first equation in \eqref{eq:Sode}. A standard comparison arguments yields
\begin{equation}\label{eq:upperbound-1-V}
V(x,x^*)\le V_1(x) = B\cdot \left(\dfrac{x}{x^*}\right)^{\frac{r}{\beta+\sigma^2}}\qquad\forall x\in [0,x^*],
\end{equation}
which proves $(i)$ in this case.
\par\medskip\par
\item If $\theta(x^*)<\gamma$ then let $(\tilde{V}_1(\cdot),\tilde{p}_1(\cdot))$ be the  solution to the 
backward Cauchy problem
\begin{align}\label{eq:pV-upper}
\begin{cases}
r\tilde{V}_1(x)&=~\displaystyle \Big(\frac{\lambda+r}{\tilde{p}_1(x)}+\sigma^2\Big)x \tilde{V}_1'(x)\,, \\[6mm]
\beta\cdot (\tilde{p}_1(x)-\gamma)&=~\displaystyle\Big(\frac{\lambda+r}{\tilde{p}_1(x)}+\sigma^2\Big) x \tilde{p}'_1(x)\,,
\end{cases}
&&\begin{cases}
\tilde{V}_1(x^*)&=~ B, \\[6mm]
\tilde{p}_1(x^*)&=~\theta(x^*).
\end{cases}
\end{align}
We have that $\tilde p_1(\cdot)$ is strictly decreasing and for $x\ne 0$ it solves the implicit equation
\begin{align*}
\tilde{p}_1(x) = \dfrac{\theta(x^*)x^*}{x}\cdot \left(\dfrac{\gamma-\tilde{p}_1(x)}{\gamma-\theta(x^*)}\right)^{1+\frac{\sigma^2}{\beta}}\qquad\textrm{ with }\qquad
\lim_{x\to 0^+}\tilde{p}_1(x)=\gamma,
\end{align*}
while $\tilde{V}_1(\cdot)$ is increasing and it can be expressed in terms of $\tilde p_1(\cdot)$ as follows
\begin{align}\label{eq:vt1}
\tilde{V}_1(x) = B\cdot \left(\dfrac{\tilde{p}_1(x)\cdot x}{\theta(x^*)\cdot x^*}\right)^{\frac{r}{\beta+\sigma^2}}\le B\cdot  \left(\dfrac{x}{\theta(x^*)x^*}\right)^{\frac{r}{\beta+\sigma^2}}
\end{align}
for all $x\in [0,x^*]$. Using \eqref{eq:pV-upper} and the implicit expression of $\tilde p_1(\cdot)$, we obtain
\begin{align*}
-1\displaystyle=&\tilde{p}'_1(x)\cdot\left[ \frac{x}{\tilde{p}_1(x)}+\left(1+\frac{\sigma^2}{\beta}\right)\cdot\frac{x}{\gamma-\tilde{p}_1(x)}\right]\\
=&\tilde{p}'_1(x)\cdot\left[ \frac{x}{\tilde{p}_1(x)}+\left(1+\frac{\sigma^2}{\beta}\right)\cdot\left(\dfrac{[\theta(x^*)x^*]^{\frac{\beta}{\beta+\sigma^2}}}{\gamma-\theta(x^*)}\right)\cdot \dfrac{x^{\frac{\sigma^2}{\beta+\sigma^2}}}{\tilde{p}_1(x)^{\frac{\beta}{\beta+\sigma^2}}}\right]\,.\\
\end{align*}
Since $\tilde{p}_1$ is strictly decreasing, we have $\tilde{p}_1'<0$, and from the above expression it follows that $\tilde{p}_1'(x)$ is increasing, and so $\tilde{p}_1''(x)>0$ for all $x\in ]0,x^*[$. 
Hence, from Lemma \ref{lemma:Hprop} (1), it holds
\begin{multline}\label{eq:sub-p1-first}
(r+\lambda)-(r+\lambda+v(x,x^*))\tilde{p}_1+H_{\xi}(x,V'(x,x^*),\tilde{p}_1)\tilde{p}_1'+\dfrac{\sigma^2x^2}{2}\tilde{p}_1''\\
\geq~(r+\lambda)- (r+\lambda+v_{\max})\tilde{p}_1+\left(\frac{\lambda+r}{\tilde{p}_1}+\sigma^2\right)x\tilde{p}_1'+\frac{\sigma^2x^2}{2}\tilde{p}_1''\\
 = \beta\cdot (\gamma-\tilde{p}_1)+\left(\frac{\lambda+r}{\tilde{p}_1}+\sigma^2\right)x\tilde{p}_1'+\frac{\sigma^2x^2}{2}\tilde{p}_1'' = \frac{\sigma^2x^2}{2}\tilde{p}_1''>0
\end{multline}
for all $x\in ]0,x^*[$.
On the other hand, recalling $\bar{x}_0=\dfrac{c'(0)}{M^*}$ in \eqref{eq:nondev-asy}, let $\bar{p}_1(\cdot)$ be the solution of the backward Cauchy problem 
\begin{align*}(r+\lambda)(\bar{p}_1(x)-1)=\displaystyle\left(\frac{\lambda+r}{\bar{p}_1(x)}+\sigma^2\right) x \bar{p}'_1(x)\qquad\mathrm{with}\qquad \bar{p}_1(\bar{x}_0)=\tilde{p}_1(\bar{x}_0)\,.\end{align*}
It is clear that
\[
\bar{p}_1(x)\ge \tilde{p}_1(x)\qquad\forall x\in [0,\bar{x}_0].
\] 

Arguing as in the case of $\tilde p_1(\cdot)$, we have that $\bar{p}_1$ is decreasing, $\displaystyle\lim_{x\to 0+}\bar{p}_1(x)=1$, and 
\begin{align*}\bar{p}_1(x)=\dfrac{\tilde{p}_1(\bar{x}_0)\bar{x}_0}{x}\cdot \left(\dfrac{1-\bar{p}_1(x)}{1-\bar{p}_1(\bar{x}_0)}\right)^{\frac{\sigma^2+\lambda+r}{\lambda+r}}\qquad\textrm{ for all }x\in [0,\bar{x}_0].\end{align*}
Moreover, $\bar{p}_1''(x)>0$ and 
\begin{equation}\label{eq:sub-p1-second}
(r+\lambda)(1-\bar{p}_1)+H_{\xi}(x,V'(x,x^*),\bar{p}_1)\bar{p}_1'+\frac{\sigma^2x^2}{2}\bar{p}_1''>0\qquad\forall x\in ]0,\bar{x}_0[.
\end{equation}
Let $p_1:[0,x^*]\to \mathbb R$ be such that
\begin{equation*}
p_1(x)=\begin{cases}
\bar{p}_1(x),&\textrm{ for all }x\in [0,\bar{x}_0]\,, \\[4mm]
\tilde{p}_1(x),&\textrm{ for all }x\in [\bar{x}_0,x^*]\,,\end{cases}
\end{equation*}
we have 
\begin{align*}p_1(0) = p(0,x^*) = 1\qquad\textrm{and}\qquad p_1(x^*) = p(x^*,x^*) = \theta(x^*).\end{align*}
Recalling that $v(x,x^*)=0$ for all $x\in [0,x_0]$, \eqref{eq:sub-p1-first} and \eqref{eq:sub-p1-second} imply that 
\[
(r+\lambda)-(r+\lambda+v(x,x^*))p_1+H_{\xi}(x,V'(x,x^*),p_1)p_1'+\dfrac{\sigma^2x^2}{2}p_1''>0
\]
for all $x\in ]0,x_0[\cup ]x_0,x^*[$. Thus, $p_1(\cdot)$ is a sub-solution of the second equation in \eqref{eq:Sode} and a standard comparison arguments yields
\begin{equation}\label{eq:upperbound-p}
p(x,x^*)\ge p_1(x)\ge \tilde{p}_1(x)\qquad\forall x\in [0,x^*]\,.
\end{equation}
To complete this step, we define 
\begin{equation*}
V_1(x)\doteq \begin{cases}
\tilde{V}_1(\bar{x}_1),&\textrm{ for } x\in [0,\bar{x}_1]\\[4mm]
\tilde{V}_1(x),&\textrm{ for } x\in [\bar{x}_1,x^*]\end{cases}
\quad\mathrm{with}\quad \bar{x}_1\doteq \min\left\{\dfrac{1}{r+\lambda},\bar{x}_0,x^*\right\}
\end{equation*}
For any $x\in ]0,\bar{x}_1[$, it holds that $V_1(x)=\tilde{V}_1(\bar{x}_1)$ and thus
\begin{equation}\label{eq:cc1}
-rV_1(x)+H(x,V_1'(x),p(x,x^*))+\frac{\sigma^2x^2}{2}V_1''(x) = -r\tilde{V}_1(x_1)<0\,.
\end{equation}
On the other hand, from Lemma \ref{lemma:Hamgrad}, the map $p\mapsto H(x,\xi,p)$ is monotone decreasing when $\xi\geq 0$ and $x\geq \frac{1}{\lambda+r}\geq \bar{x}_1$. 
Thus, recalling \eqref{eq:upperbound-p}, \eqref{eq:pV-upper}, and Lemma \ref{lemma:Hprop} (1), we have that  all $x\in [\bar{x}_1,x^*]$, it holds
\begin{eqnarray*}
-rV_1(x)+H(x,V_1'(x),p(x,x^*))&\leq&-rV_1(x)+H(x,V_1'(x),\tilde{p}_1(x))\\
&=&-r\tilde{V}_1(x)+H(x,\tilde{V}_1'(x),\tilde{p}_1(x))\\
&\leq&-r\tilde{V}_1(x)+\left(\dfrac{r+\lambda}{\tilde{p}_1(x)}+\sigma^2\right)x\tilde{V}_1'(x) = 0.
\end{eqnarray*}
Differentiating both sides of the first ODE in \eqref{eq:pV-upper}, we obtain
\[
\left(r-\sigma^2-\frac{\lambda+r}{\tilde{p}_1(x)} +\frac{(\lambda+r)\tilde{p}_1'(x)}{\tilde{p}_1^2(x)}x\right)\cdot \tilde{V}_1'(x) = \left(\frac{\lambda+r}{\tilde{p}_1(x)}+\sigma^2\right)x \tilde{V}_1''(x)
\]
and it yields
\[\tilde{V}_1''(x)<0\qquad\forall x\in \,]0,x^*[\,.\]
Hence, for all $x\in [\bar{x_1},x^*[$, it holds
\[
-rV_1(x)+H(x,V_1'(x),p(x,x^*))+\frac{\sigma^2x^2}{2}V_1''(x)\le \frac{\sigma^2x^2}{2}V_1''(x) = \frac{\sigma^2x^2}{2}\tilde{V}_1''(x)<0
\]
Hence, $V_1$ is a super-solution of the first equation in \eqref{eq:Sode} and a standard comparison arguments yields
\begin{align*}
V(x,x^*)\le V_1(x) = \tilde{V}_1(x)\quad\textrm{ for all }x\in [\bar{x}_1,x^*].
\end{align*}
From \eqref{eq:asp0} and \eqref{eq:vt1}, we obtain that 
\[
\limsup_{x^*\to\infty}~V(x,x^*)\le \limsup_{x^*\to\infty}~B\cdot  \left(\dfrac{x}{\theta(x^*)x^*}\right)^{\frac{r}{\beta+\sigma^2}} = 0\quad\forall x\in [\bar{x}_1,x^*],
\]
and the monotone increasing of $V(\cdot,x^*)$ yields \eqref{eq:stoch-ponzi}.

\end{itemize}

\par\medskip\par

\noindent \textbf{Step 2.} We now prove (ii). Suppose that 
\begin{equation}\label{eq:asp1}
\limsup_{s\to+\infty}\theta(s)s = C_1<+\infty\,.
\end{equation}
We introduce an intermediate point 
\begin{equation}\label{eq:x2}
\bar{x}_2\doteq \dfrac{1}{r-\mu}+\theta(x^*)x^*
\end{equation}
which depends on $x^*$. Define
\[p_2(x) = \begin{cases} 1,&\textrm{ for }x\in [0,\bar{x}_2],\\[4mm]
\displaystyle \dfrac{\bar{x}_2}{x}=\left(\dfrac{1}{r-\mu}+\theta(x^*)x^*\right)\cdot \dfrac{1}{x},&\textrm{ for }x\in [\bar{x}_2,x^*],
\end{cases}\]
we have
\begin{align*}
 p_2(0) = 1 = p(0,x^*),\qquad p_2(x^*) = \dfrac{1}{(r-\mu)x^*}+\theta(x^*)>p(x,x^*),
\end{align*}
and
\begin{align*}
p'_2(x) = -\dfrac{p_2(x)}{x},\qquad p''_2(x) = \dfrac{2p_2(x)}{x^2}\qquad\forall x\in  [\bar{x}_2,x^*].
\end{align*}
For $x\in ]0,\bar{x}_2[$, it holds
\begin{multline*}
(r+\lambda)-(r+\lambda+v(x,x^*))p_2(x)+H_{\xi}(x,V'(x,x^*),p_2(x))\cdot p'_2(x)+\dfrac{\sigma^2 x^2}{2} p_2''(x)\\  = -v(x,x^*)p_2(x) = -v(x,x^*)<0\,.
\end{multline*}
On the other hand, recalling Lemma \ref{lemma:Hamgrad}, we have
\begin{multline*}
(r+\lambda)-(r+\lambda+v(x,x^*)) p_2(x)+H_{\xi}(x,V'(x,x^*), p_2(x))\cdot p'_2(x)+\dfrac{\sigma^2 x^2}{2} p_2''(x)\\
 = \dfrac{{u^*(V'(x,x^*),p_2(x))}}{x}-(r-\mu)\cdot p_2(x)\le \dfrac{1}{x}-(r-\mu)\cdot p_2(x)<0
\end{multline*}
for all $x\in ]\bar{x}_2,x^*[$. Therefore, $ p_2(\cdot)$ is a super-solution of the first equation in \eqref{eq:Sode} and 
\begin{align}\label{eq:p2-b}
p(x,x^*)\le  p_2(x)\qquad\textrm{ for all }x\in [0,x^*]\,.
\end{align}
To construct $V_2$, we define
\begin{equation}\label{eq:x3}
\bar{x}_3~:=~\dfrac{1+(r+v_{\max})\cdot \bar{x}_2}{r}>\dfrac{1}{r+\lambda}.
\end{equation}
Notice that for $x^*$ sufficiently large, it holds $\bar{x}_3<x^*$. In this case, we set
\[
V_2(x) = \begin{cases}
0,&\textrm{ for }x\in [0,\bar{x}_3],\\[4mm]
\left[\dfrac{\bar{x}_2}{\bar{x}_3} - p_2(x)\right]\cdot B = B\bar{x}_2\cdot\left(\dfrac{1}{\bar{x}_3}-\dfrac{1}{x}\right),&\textrm{ for }x\in [\bar{x}_3,x^*].\end{cases}
\]
For every $x\in ]0,\bar{x}_3[$, it holds
\begin{equation}\label{eq:ine1}
-rV_2+H(x,V'_2,p(x,x^*))+\frac{\sigma^2x^2}{2}V_2'' = H(x,0,p(x,x^*)) = 0\,.
\end{equation}
For every $x\in ]\bar{x}_3,x^*[$, we have
\[
V_2'(x) = B\cdot \dfrac{p_2(x)}{x}>0,\qquad V''_2(x) = -2B\cdot \frac{p_2(x)}{x^2}<0\,.
\]
Thus, recalling \eqref{eq:x3} and Lemma \ref{lemma:Hprop} (1), we estimate 
\begin{eqnarray*}
-rV_2+H(x,V_2',p_2)&\geq&-rV_2+\left(\frac{(\lambda+r) x-1}{p_2}+(\sigma^2-\lambda-\mu-v_{\max})x\right)V'_2\\
&=&-rV_2+B\cdot \left(\lambda+r-\frac{1}{x}+(\sigma^2-\lambda-\mu-v_{\max})\cdot p_2\right)\\
&\geq&-rV_2+B\cdot \left(r-\frac{1}{x}+(\sigma^2-\mu-v_{\max})\cdot p_2\right)\\
&=&B\cdot \left(r-\frac{1}{x}-\frac{r\bar{x}_2}{\bar{x}_3}+(r-\mu-v_{\max})\cdot p_2\right)+B\sigma^2p_2\\
&\geq&B\cdot \left(r-\frac{1}{x}-\frac{r\bar{x}_2}{\bar{x}_3}-v_{\max}\cdot p_2\right)-\frac{\sigma^2x^2}{2}\cdot V''_2\\
&\geq&B\cdot\frac{{r\bar{x}_3}-(1+(r+v_{\max})\cdot \bar{x}_2)}{\bar{x}_3}-\frac{\sigma^2}{2}\cdot V''_2 = -\frac{\sigma^2x^2}{2}\cdot V''_2
\end{eqnarray*}
and it yields 
\[-rV_2+H(x,V_2',p_2)+\frac{\sigma^2x^2}{2}\cdot V''_2\ge 0\qquad\forall x\in [\bar{x}_3,x^*].\]

From Lemma \ref{lemma:Hamgrad} and \eqref{eq:x3}, the map $p\to H(x,V'_2(x),p)$ is monotone decreasing  on $[0,1]$, for all $x\in  [\bar{x}_3,x^*]$. Recalling \eqref{eq:p2-b}, we get
\[-rV_2+H(x,V_2',p(x,x^*))+\frac{\sigma^2x^2}{2} V_2''\ge 0\qquad\forall x\in ]\bar{x}_3,x^*[.\]
Since $V_2(0)=V(x,x^*)$ and $V_2(x^*)<B=V(x,x^*)$, together with \eqref{eq:ine1}, the function $V_2$ is a sub-solution of the first equation in \eqref{eq:Sode} and 
\begin{align*}
V(x,x^*)\ge& V_2(x) =  B\bar{x}_2~\cdot~ \left[\dfrac{1}{\bar{x}_3}-\dfrac{1}{x}\right]\ge \frac{B\bar{x}_2}{2\bar{x}_3}\\
=&\frac{rB}{2}\cdot \frac{1}{r+v_{\max}+\frac{1}{\bar{x}_2}}\ge \frac{rB}{2(2r+v_{\max}-\mu)}
\end{align*}
for all  $x\in [2\bar{x}_3,x^*]$. Finally, recalling \eqref{eq:asp1}, \eqref{eq:x2}, and \eqref{eq:x3}, we have  
\[\limsup_{x^*\to\infty}\bar{x}_3 = \frac{(r+v_{\max}+(1+C_1(r+v_{\max}))(r-\mu))}{r(r-\mu)}\]
and it yields $(ii)$.
\end{proof}

\section{The deterministic case $\sigma=0$}\label{sec:det-model}
In the case $\sigma=0$, the stochastic control system \eqref{eq:debt-GDP-ev} reduces to the deterministic one
\begin{equation}\label{eq:dode1}
\dot{x}(t)=\left(\dfrac{\lambda+r}{ p(t)} -\lambda  -\mu - v(t)\right) x(t) - \dfrac{u(t)}{ p(t)}.
\end{equation}
Here the control $u(t)$ is assumed to be in $[0,1]$ for all $t\geq 0$. The Debt Management Problem can be formulated as follows.\par\medskip\par
\begin{itemize}
\item [\textbf{(DMP)}]\emph{Given an initial value $x(0)=x_0\in [0, x^*]$ of the DTI, minimize
\begin{equation}\label{eq:cost-det}
\int_0^{T_b} e^{-rt} [L(u(t))+c(v(t))]\, dt + e^{-rT_b} B,
\end{equation}
subject to the dynamics \eqref{eq:dode1},
where the bankruptcy time $T_b$ is defined as in \eqref{eq:bankruptcy-time},
while the discount bond price 
\begin{multline}\label{eq:covec-det}
p(t)=\int_t^{T_b}(r+\lambda)\exp\left\{-\int_0^s \bigl(\lambda +r+v(\tau)\bigr)\,d\tau\right\}\,ds+\\
+\exp\left\{-\int^{T_b}_{t}(r+\lambda+v(s))\,ds\right\}\theta(x^*).
\end{multline}}
\end{itemize}
As in the stochastic case, we will assume that \textbf{(A1)}-\textbf{(A2)} hold. 
Since in this case the optimal feedback control $u^*,v^*$ and the corresponding functions $V^*,p^*$ may be
nonsmooth, a concept of equilibrium solution should be more carefully defined.
\begin{definition}[\textbf{Equilibrium solution in feedback form}]\label{def:eqsol} 
\emph{A couple of piecewise Lipschitz continuous  functions $(u^*(\cdot),v^*(\cdot))$ and l.s.c. 
$p^*(\cdot)$ provide an  equilibrium solution to the debt management problem (DMP), with 
continuous value function $V^*(\cdot)$, if
\begin{itemize}
\item[(i)] For every $x_0\in [0,x^*]$, $V^*$ is the minimum cost for the optimal control problem 
\begin{equation}\label{eq:cost-feed}
\textrm{ minimize: }\int_0^{T_b} e^{-rt}[L(u(t))+c(v(t))]\, dt + e^{-r T_b} B,
\end{equation}
subject to 
\begin{equation}\label{eq:cont-feed}
\dot{x}(t)=\left(\dfrac{\lambda+r}{p^*(x(t))}-\lambda-\mu-v(t)\right)x(t)-\dfrac{u(t)}{p^*(x(t))},\hspace{1cm} x(0)=x_0,
\end{equation}
where $u:[0,+\infty[\to [0,1]$ and $v:[0,+\infty[\to [0,+\infty[$ are measurable functions.
Moreover, every Carath\'eodory solution of \eqref{eq:cont-feed} with $(u(t),v(t))= (u^*(x(t)),v^*(x(t)))$ is optimal.
\item[(ii)] For every $x_0\in [0,x^*]$, there exists at least one solution $t\mapsto x(t)$ of the Cauchy problem
\begin{equation}
\dot{x}(t)=\left(\dfrac{\lambda+r}{p^*(x(t))} -\lambda -\mu-v^*(x(t))\right) x(t) - \dfrac{u^*(x(t))}{p^*(x(t))}\,,\qquad\qquad x(0)=x_0,
\end{equation}
such that
\begin{multline}
p^*(x_0)=\int_0^{T_b}(r+\lambda)\exp\left\{-\int_0^t \bigl(\lambda +r+v^*(x(s)\bigr)\,ds\right\}\,dt+\\
+\exp\left\{-\int^{T_b}_{0}(r+\lambda+v^*(x(t)))\,dt\right\}\theta(x^*),\end{multline}
with $T_b$ as in \eqref{eq:bankruptcy-time}.
\end{itemize}}
\end{definition}

\subsection{System of first order Hamilton-Jacobi equations} 
In the deterministic case, the Debt Management Problem \eqref{eq:dode1}-\eqref{eq:covec-det} leads to the following implicit system of first order ODEs
\begin{equation}\label{eq:Sdode}
\begin{cases}
rV(x)=H(x,V'(x),p(x))\,,\\[4mm]
(r+\lambda+v(x))p(x)-(r+\lambda)=H_{\xi}(x,V'(x),p(x))\cdot p'(x)\,,\\[6mm]
v(x)=\displaystyle\underset{\omega\geq 0}{\mathrm{argmin}}\{c(\omega)-\omega xV'(x)\}\,,
\end{cases}\end{equation}
with the boundary conditions
\begin{align}\label{eq:bdc}
\begin{cases}V(0)=0,\\ \\ V(x^*)=B,\end{cases}&&
\begin{cases}p(0)=1,\\ \\ p(x^*)=\theta(x^*).\end{cases}
\end{align}
The Hamiltonian function \eqref{eq:Hamiltonian} reduces to
\begin{multline}\label{eq:detHamil}
H(x,\xi,p)\doteq\min_{u\in [0,1]}\left\{L(u)-u\,\dfrac{\xi}{p}\right\}+\min_{v\geq 0}\Big\{c(v)-vx\xi\Big\}
+ \left(\dfrac{\lambda+r}{p}-\lambda-\mu\right)x\,\xi.
\end{multline}
Therefore, we have (see Appendix \ref{app:convex} for the notation)
\begin{equation}\label{eq:detHamil-2}
-H(x,\xi,p)\doteq L^\circ\left(\dfrac{\xi}{p}\right)+c^\circ(x\xi)-\left(\dfrac{\lambda+r}{p}-\lambda-\mu\right)x\,\xi,
\end{equation}
and we notice that for fixed $x>0$, $p\in ]0,1]$ the map $\xi\mapsto -H(x,\xi,p)$ is convex and lower semicontinuous.
Here $L^\circ,c^\circ$ are  the convex conjugate of $L$ and $c$.

Given $x>0$, $p\in ]0,1]$, $\xi\ge 0$, 
we denote by $u^*(\xi,p)\in [0,1]$ and $v^*(x,\xi)\in [0,+\infty[$ the unique elements of 
$\partial L^\circ\left(\dfrac{\xi}{p}\right)$ and $\partial c^\circ(x\xi)$, respectively, provided by Lemma \ref{lemma:subdiffcost}.
\begin{align*}
u^*(\xi,p)&\doteq~\underset{u\in[0,1]}{\mathrm{argmin}}\left\{L(u) - u\,\dfrac{\xi}{p}\right\}=\begin{cases}0,&\textrm{ if }0\le \xi< pL'(0),\\ (L')^{-1}(\xi/p),&\textrm{ if }\xi\ge pL'(0)>0,\end{cases}\\
v^*(x,\xi)&\doteq~\underset{v\geq 0}{\mathrm{argmin}}\Big\{c(v)-vx\xi\Big\}=\begin{cases}0,&\textrm{ if }0\le x\xi< c'(0),\\ (c')^{-1}(x\xi),&\textrm{ if }x\xi\ge c'(0)>0.\end{cases}
\end{align*}
In particular, 
\begin{itemize}
\item for every $p\in]0,1]$ the map $\xi\mapsto u^*(\xi,p)$ is strictly increasing in $[pL'(0),+\infty[$, and $u^*(\cdot, p)\equiv 0$ in $[0,pL'(0)]$; 
\item for every $\xi\ge 0$ the map $p\mapsto u^*(\xi,p)$ is strictly decreasing in $[\xi/L'(0),1]$, and $u^*(\xi, \cdot)\equiv 0$ in $[0,\xi/L'(0)]$; 
\item for every $\xi>0$ the map $x\mapsto v^*(x,\xi)$ is strictly increasing in $[c'(0)/\xi,+\infty[$, and $v^*(\cdot,\xi)\equiv 0$ in $[0,c'(0)/\xi]$;
\item for every $x>0$ the map $\xi\mapsto v^*(x,\xi)$ is strictly increasing in $[c'(0)/x,+\infty[$, and $v^*(x,\cdot)\equiv 0$ in $[0,c'(0)/x]$.
\end{itemize}
\noindent It is proved in  Lemma \ref{lemma:Hamgrad} that the gradient of the Hamiltonian function $H(\cdot)$ 
at points $(x,\xi,p)\in [0,+\infty[\times[0,+\infty[\times]0,1]$ can be expressed in terms of 
$u^*(\xi,p)$ and $v^*(x,\xi)$ by 
\begin{equation}\label{eq:gradH}\begin{cases}
\displaystyle H_{x}(x,\xi,p)=& \displaystyle \Big[(\lambda+r)-p(\lambda +\mu +v^*(x,\xi))\Big]\cdot \frac{\xi}{p},\\
\displaystyle H_{\xi}(x,\xi,p)=& \displaystyle \frac{1}{p}\cdot \Big[x\big((\lambda+r)-p(\lambda +\mu +v^*(x,\xi))\big)-u^*(\xi,p)\Big],\\
\displaystyle H_{p}(x,\xi,p)=& \displaystyle (u^*(\xi,p)-x(\lambda +r))\cdot \frac{\xi}{p^2}.
\end{cases}
\end{equation}
The following Lemma will catch some relevant properties of $H(\cdot)$ needed to study the system \eqref{eq:Sdode}.
\par\medskip\par
\begin{lemma}\label{lemma:nonincmax}
Let $x\geq 0$ and $0<p\le 1$ be fixed, and set
\[H^{\max}(x,p)\doteq \max_{\xi\geq 0}H(x,\xi,p).\] 
Then 
\begin{enumerate}
\item there exists $\xi^\sharp(x,p)>0$ such that, given $\eta>0$, the equation $r\eta=H(x,\xi,p)$ admits
\begin{itemize}
\item no solutions $\xi\in[0,+\infty)$ if $r\eta> H^{\max}(x,p)$,
\item $\xi^{\sharp}(x,p)$ as unique solution  if $r\eta=H^{\max}(x,p)$, 
\item exactly two distinct solutions $\{F^{-}(x,\eta,p),F^{+}(x,\eta,p)\}$ 
with \[0<F^{-}(x,\eta,p)<\xi^{\sharp}(x,p)<F^{+}(x,\eta,p)\] if $0<r\eta<H^{\max}(x,p)$,
\end{itemize}
\par\medskip\par
\item we extend the definition of $\eta\mapsto F^{\pm}(x,\eta,p)$ by setting 
\[F^{\pm}\left(x,\dfrac{1}{r}\,H^{\max}(x,p),p\right)=\xi^{\sharp}(x,p),\]
thus for fixed $x>0$ , $p\in ]0,1]$, the maps $\eta\mapsto F^{-}(x,\eta,p)$ and $\eta\mapsto F^{+}(x,\eta,p)$
are respectively strictly increasing and strictly decreasing in $[0,H^{\max}(x,p)/r]$.
\item for all $0<\eta<H^{\max}(x,p)/r$ with $x>0$ and $p\in]0,1]$, we have 
\[\dfrac{\partial}{\partial \eta} F^{\pm}(x,\eta,p)=\dfrac{r}{H_\xi(x,F^{\pm}(x,\eta,p),p)},\]
\item The map $p\mapsto H^{\max}(x,p)$ is strictly decreasing on $]0,1]$ for every fixed $x\in ]0,x^*[$.
\end{enumerate}
\end{lemma}
\begin{proof}
Since for all fixed $x>0$, $0<p\le 1$ we have that $\xi\mapsto H(x,\xi,p)$ is the minimum of a family of affine functions of $\xi$, 
we have that the map $\xi\mapsto H(x,\xi,p)$ is concave down. 
Recalling \eqref{eq:gradH}, and the monotonicity properties of $u^*(\cdot,p)$ and $v^*(x,\cdot)$, since
\begin{itemize}
\item $H_{\xi}(x,\xi,p)=H_{\xi}(x,0,p)>0,\textrm{ for all }\xi\in [0,\min\{pL'(0),c'(0)/x\}]$,
\item $\xi\mapsto H_{\xi}(x,\xi,p),\textrm{ is strictly decreasing for all }\xi>\min\{pL'(0),c'(0)/x\}$,
\item $\displaystyle\lim_{\xi\to+\infty}H_{\xi}(x,\xi,p)=-\infty$,
\end{itemize}
we have that $\xi\mapsto H_{\xi}(x,\xi,p)$ vanishes in at most one point in $[0,+\infty)$, 
so $\xi\mapsto H(x,\xi,p)$ reaches its maximum value $H^{\max}(x,p)$ on $[0,+\infty)$ at a unique point $\xi^{\sharp}(x,p)$,
moreover it is strictly increasing for $0<\xi<\xi^{\sharp}(x,p)$ and strictly decreasing for $\xi>\xi^{\sharp}(x,p)$, with 
$\xi^{\sharp}(x,p)\ge \min\{pL'(0),c'(0)/x\}$.\par\medskip\par
We define 
\begin{itemize}
\item the strictly increasing map $\eta\mapsto F^-(x,\eta,p)$, for $0<\eta<H^{\max}(x,p)/r$, to be the inverse of $\xi\mapsto \frac 1r H(x,\xi,p)$ for $0<\xi<\xi^{\sharp}(x,p)$;
\item the strictly decreasing map $\eta\mapsto F^+(x,\eta,p)$, for $0<\eta<H^{\max}(x,p)/r$, to be the inverse of $\xi\mapsto \frac 1r H(x,\xi,p)$ for $\xi>\xi^{\sharp}(x,p)$.
\end{itemize}
Set
\begin{align*}
u^{\sharp}(x,p)\doteq u^*(\xi^{\sharp}(x,p),p),&&v^{\sharp}(x,p)\doteq v^*(x,\xi^{\sharp}(x,p)).
\end{align*}
Recalling \eqref{eq:gradH}, we have
\begin{align}
\label{eq:Rmax}u^{\sharp}(x,p)&=~\big[(\lambda+r)-(\lambda+\mu+v^{\sharp}(x,p))p\big]\cdot x,\\
\label{eq:Hmax}H^{\max}(x,p)&=~L(u^{\sharp}(x,p))+c(v^{\sharp}(x,p)).
\end{align}
Moreover, 
\begin{itemize}
\item if $\xi^{\sharp}(x,p)\ge pL'(0)$ we have
\begin{equation}
\label{eq:Xisharp}
\xi^{\sharp}(x,p)=pL'(u^{\sharp}(x,p))=pL'\left(\big[(\lambda+r)-(\lambda+\mu+v^{\sharp}(x,p))p\big]\cdot x\right),
\end{equation}
\item if $\xi^{\sharp}(x,p)\ge c'(0)/x$ we have
\begin{equation}
\xi^{\sharp}(x,p)=c'(v^{\sharp}(x,p))/x,
\end{equation}
\end{itemize}
Conversely, assume that given $x>0$, $p\in ]0,1]$, we have 
\[u^*(\xi,p)=x\big((\lambda+r)-p(\lambda +\mu +v^*(x,\xi))\big),\]
then $\xi=\xi^{\sharp}(x,p)$, $v^*(x,\xi)=v^\sharp(x,p)$, $u^*(\xi,p)=u^\sharp(x,p)$.
This follows from the fact that $H_\xi(x,\xi,p)=0$ iff $\xi=\xi^\sharp(x,p)$.
\par\medskip\par
For any fixed $x\geq 0$ and $0<p\le 1$, given $\eta>0$ we consider the equation $r\eta=H(x,\xi,p)$,
and all the statements (1-2-3) follows by applying $F^{\pm}(x,\cdot,p)$ to it. To prove item (4), we notice that
\[\dfrac{d}{dp} H^{\max}(x,p)=\dfrac{d}{dp}H(x,\xi^\sharp(x,p),p)=H_p(x,\xi^\sharp(x,p),p).\]
Recalling \eqref{eq:gradH}, we have
\begin{eqnarray*}
H_p(x,\xi^{\sharp}(x,p),p)&=&\left[{u}^{\sharp}(x,p)-(r+\lambda)x\right]\cdot \dfrac{\xi^{\sharp}(x,p)}{p}\\
&=&-(\lambda+\mu+v^{\sharp}(x,p))x\xi^\sharp(x,p)<0\,,
\end{eqnarray*}
since for $x,p\ne 0$ we have $\xi^\sharp(x,p)>0$.
\end{proof}

\begin{definition}[Normal form of the system]
Given $x>0$, $0<p\le 1$, $0<r\eta\le H^{\max}(x,p)$ we define the maps
\begin{equation}\label{eq:pnormal}
G^{\pm}(x,\eta,p)=\dfrac{(r+\lambda+v^*(x,F^{\pm}(x,\eta,p)))p-(r+\lambda)}{H_{\xi}(x,F^{\pm}(x,\eta,p),p)}.
\end{equation}
Notice that if $rV(x)>H^{\max}(x,p)$, then the first equation of \eqref{eq:Sdode} has no solution.
Otherwise, if $0<rV(x)<H^{\max}(x,p)$ this equation splits into 
\begin{align*}
\begin{cases}V'(x)=F^{-}(x,V(x),p(x)),\\ p'(x)=G^{-}(x,V(x),p(x)),\end{cases}&&\textrm{ or }&&
\begin{cases}V'(x)=F^{+}(x,V(x),p(x)),\\ p'(x)=G^{+}(x,V(x),p(x)).\end{cases}
\end{align*}
\end{definition}

\begin{figure}[ht]
\centering
\begin{tikzpicture}[line cap=round,line join=round,>=triangle 45,x=0.3 \textwidth,y=0.3 \textwidth]
\draw[->,color=black] (-0.10,0) -- (2.40,0);\draw[->,color=black] (0,-0.20) -- (0,1);
\clip(-0.7,-0.20) rectangle (2.40,1.2);
\draw [samples=50,rotate around={-15:(0,0)},domain=0:1.737)] plot (\x,{1-(\x-1)^2});
\draw [dotted] (0.40,0.395)-- (0.40,0);\draw [dotted] (1.09,0.72)-- (1.09,0);
\draw [dotted] (1.61,0.395)-- (1.61,0);\draw [dotted] (1.61,0.395)-- (0,0.395);
\draw [dotted] (1.09,0.725)-- (0,0.725);
\node [below, color=black] at (0.40,0) {$F^{-}(x,\eta,p)$};
\node [below, color=black] at (1.09,0) {$\xi^\sharp(x,p)$};

\node [below, color=black] at (1.61,0) {$F^{+}(x,\eta,p)$};
\node [anchor=north east, color=black] at (0,0) {$O$};
\draw [color=black, fill=black] (0.40,0) circle (1.5pt);
\draw [color=black, fill=black] (1.09,0) circle (1.5pt);
\draw [color=black, fill=black] (1.61,0) circle (1.5pt);
\node [left, color=black] at (0,0.395) {$r\eta$};
\node [left, color=black] at (0,0.725)  {$H^{\max}(x,p)$};
\draw [color=black, fill=black] (0,0.395) circle (1.5pt);
\draw [color=black, fill=black] (0,0.725) circle (1.5pt);
\node [anchor=north east, color=black] at (2.40,0) {$\xi$};
\end{tikzpicture}
\label{fig:g53}
\caption{For $x\ge 0$, $p\in]0,1]$, the function $\xi\mapsto H(x,\xi,p)$ has a unique global maximum
$H^{\max}(x,p)$ attained at $\xi=\xi^\sharp(x,p)$. For $0<r\eta\leq H^{\max}$, the values $F^-(x,\eta, p)\leq \xi^\sharp(x,p)\leq F^+(x,\eta,p)$ are well defined.
Moreover, $F^{\pm}(x,\frac1rH^{\max}(x,p),p)=\xi^\sharp(x,p)$.}
\end{figure}
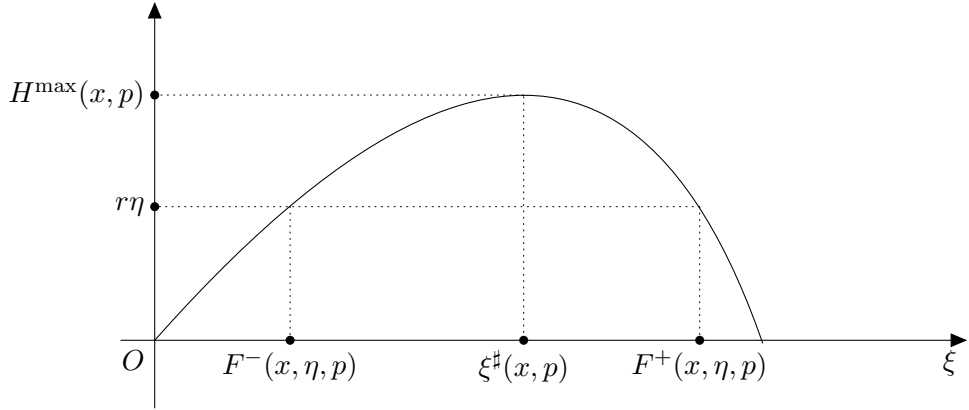

\begin{remark} 
Recalling \eqref{eq:dode1} and \eqref{eq:Rmax}, we observe that 
\begin{itemize}
\item The value $V'(x)=F^+(x,V(x),p)\geq \xi^\sharp(x,p)$ corresponds to the choice of an optimal control such that $\dot x(t)<0$. The total debt-to-ratio is descreasing.
\item The value $V'(x)=F^-(x,V(x),p)\leq \xi^\sharp(x,p)$ corresponds to the choice of an optimal control such that $\dot x(x)>0$. The total debt-to-ratio is increasing.
\item When $rV(x)= H^{\max}(x,p)$, then the value 
\[V'(x)=F^+(x,V(x),p)=F^-(x,V'(x),p)=\xi^\sharp(x,p)\] 
corresponds to the unique control strategy such that $\dot x(t)=0$. 
\end{itemize}
\end{remark}

\begin{remark}
We notice that if $0\le x\xi<\min\{xpL'(0),c'(0)\}$, since $u^*=v^*=0$, we have
\[\xi=F^{-}(x,\eta,p)=\dfrac{pr\eta}{(\lambda+r-p(\lambda+\mu))x},\]
in particular, if $0\le x\xi<\min\{xpL'(0),c'(0)\}$ we have that $\eta\mapsto F^{-}(x,\eta,p)$ is Lipschitz continuous, uniformly for $(x,p)\in [x_1,x^*]\times [p_1,1]$, 
for all $x_1\in]0,x^*]$, $p_1\in ]0,1]$.
If $x\xi>\min\{xpL'(0),c'(0)\}$, we have instead
\[H_{\xi\xi}(x,\xi,p)\le-\dfrac{1}{p}\min\left\{\dfrac{1}{p L''(u^*(\xi,p))}, \dfrac{x^2p}{c''(v^*(x,\xi))}\right\}.\]
\end{remark}

\begin{lemma}\label{lemma:holder}
Given $x_1\in]0,x^*]$, $p_1\in ]0,1]$, there exists a constant $C=C(x_1,p_1)$ such that
\[|F^{-}(x,\eta_1,p)-F^{-}(x,\eta_2,p)|\le C\cdot |\eta_1-\eta_2|^{1/2},\]
for all $x\in [x_1,x^*]$, $p\in [p_1,1]$, $0<\eta_1,\eta_2\le \frac {1}{r} H^{\max}(x,p)$.
\end{lemma}
\begin{proof}
We distinguish two cases:
\begin{enumerate}
\item if $0\le x\xi<\min\{xpL'(0),c'(0)\}$, since $u^*=v^*=0$, we have
\[\xi=F^{-}(x,\eta,p)=\dfrac{pr\eta}{(\lambda+r-p(\lambda+\mu))x},\]
and so
\begin{align*}
|F^{-}(x,\eta_1,p)-F^{-}(x,\eta_2,p)|~&\le~\dfrac{pr}{(\lambda+r-p(\lambda+\mu))x} |\eta_1-\eta_2|\\
~&\le~\dfrac{\sqrt{2B}r}{(r-\mu)x_1} |\eta_1-\eta_2|^{1/2}.
\end{align*}
for all $x\in [x_1,x^*]$, $p\in [p_1,1]$, $0<\eta_1,\eta_2\le \frac {1}{r} H^{\max}(x,p)$.\par\medskip\par
\item If $x\xi>\min\{xpL'(0),c'(0)\}$, we have instead
\[H_{\xi\xi}(x,\xi,p)\le  -\dfrac{1}{p}\min\left\{\dfrac{1}{p L''(u^*(x,\xi,p))}, \dfrac{x^2p}{c''(v^*(x,\xi))}\right\},\]
thus, recalling that by assumption we have $L''(u)\ge \delta_0$ and $c''(v)\ge \delta_0$ for $0<u<1$ and $v\ge 0$,
we obtain
\[-H_{\xi\xi}(x,\xi,p)\ge \dfrac{\min\{1,x_1^2p_1\}}{\delta_0}.\]
By applying Lemma \ref{lemma:holder-gen} to $f(\cdot)=-\frac{1}{r}H(x,\cdot,p)$, we have
\[|F^{-}(x,\eta_1,p)-F^{-}(x,\eta_2,p)|\le \sqrt{\dfrac{2r\delta_0}{\min\{1,x_1^2p_1\}}}|\eta_2-\eta_1|^{1/2}.\]
\end{enumerate}
The proof is complete by choosing $C(x_1,p_1)\doteq \sqrt{\dfrac{2r\delta_0}{\min\{1,x_1^2p_1\}}}+\dfrac{\sqrt{2B}r}{(r-\mu)x_1}$.
\end{proof}
In the next two subsections, we will provide a detail analysis on  the existence of a solution to the system of Hamilton-Jacobi equation \eqref{eq:Sdode} 
which yields an equilibrium solution to the Debt Management Problem \eqref{eq:dode1}-\eqref{eq:covec-det}.

\subsection{Constant strategies} 
We begin our analysis from the control strategies keeping the DTI constant in time, i.e., such that
the corresponding solution $x(\cdot)$ of \eqref{eq:dode1} is constant. In this case, there is no bankruptcy risk, i.e., $T_b=+\infty$.
\begin{definition}[Constant strategies]\label{def:const-strat}
Let $\bar x>0$ be given.
We say that a pair $(\bar u,\bar v)\in [0,1[\times [0,+\infty[$ is a constant strategy for $\bar x$ if
\[\begin{cases}
\left[\left(\dfrac{\lambda+r}{\bar p}-\lambda-\mu -\bar v\right)\bar x - \dfrac{\bar u}{\bar p}\right]=0,\\ \\
\bar p=\dfrac{r+\lambda}{r+\lambda + \bar v},
\end{cases}\]
where the second relation comes from taking $T_b=+\infty$ in \eqref{eq:covec-det}.
\end{definition}
\par\medskip\par
From these equations, if a couple $(\bar u,\bar v)\in [0,1[\times [0,+\infty[$ is a constant strategy then it holds $(r+\lambda)(r-\mu)\bar x=(r+\lambda+\bar v)\bar u$. 
In this case, the borrower will never go bankrupt and thus the cost of this strategy in \eqref{eq:cost-det} is computed by 
\begin{align*}
\dfrac{1}{r}\cdot \Big[L(\bar u)+c(\bar v)\Big]&=~\dfrac{1}{r}\cdot \left[L\left(\dfrac{(r+\lambda)(r-\mu)\bar x}{r+\lambda+\bar v}\right) + c\left(\bar v\right)\right]\\
&=~\dfrac{1}{r}\cdot \left[L\left((r-\mu)\bar x\cdot\bar p\right) + c\left(\left(1-\dfrac{1}{\bar p}\right)(r+\lambda)\right)\right].\end{align*}
We notice that if $\bar x(r-\mu)>1$, we must have $\bar v>1$ and $\bar p<1$, in particular if DTI is sufficiently large, every constant strategy needs to implement currency devaluation,
with a consequently drop of $p$. A more precise estimate will be provided in Proposition \ref{prop:nondev}.\par\medskip\par
We are now interested in the minimum cost of a strategy keeping the debt constant. To this aim, we first characterize the cost of a constant strategy in terms of the variables $x,p$.

\begin{lemma}\label{lemma:cconvH} 
Given any $(x,p)\in ]0,+\infty[\times ]0,1]$, we have
\begin{equation}\label{eq:convH}H^{\max}(x,p)=\min\Big\{L(u)+c(v):\, u\in [0,1],\,v\ge 0,\,u=\big[(\lambda+r)-(\lambda+\mu+v)p\big]\cdot x\Big\}.
\end{equation}
Moreover, $(\hat u,\hat v)$ realizes the minimum in the right hand side of \eqref{eq:convH} if and only if
\[
\begin{cases}
\displaystyle c(\hat v)+px\hat v\xi^\sharp(x,p)&=~\displaystyle\min_{\zeta\ge 0}\left\{px\xi^\sharp(x,p)\zeta+c(\zeta)\right\},\\ \\
\displaystyle L(\hat u)+\hat u\xi^\sharp(x,p)&=~\displaystyle\min_{u\in[0,1]}\left\{\xi^{\sharp}(x,p)u+L(u)\right\}.
\end{cases}
\]
\end{lemma}
\begin{proof} 
Set $F(v):=f(v)+g(\Lambda v)$ where $f(\zeta)=c(\zeta)$ for $\zeta\ge 0$ and $f(\zeta)=+\infty$ if $\zeta<0$, $C(x,p)=\big[(\lambda+r)-(\lambda+\mu)p\big]\cdot x$,
$g(\zeta)=L(C(x,p)+\zeta)$ if $C(x,p)+\zeta\in [0,1]$
and $g(\zeta)=+\infty$ if $C(x,p)+\zeta\notin [0,1]$, and $\Lambda=-xp$.
By standard argument in convex analysis (see e.g. Theorem 4.2 and Remark 4.2 p. 60 of \cite{ET}), denoted by $f^\circ$, $g^\circ$ the convex conjugates of $f,g$ respectively, we have 
\begin{align*}
\inf_{v\in\mathbb R} F(v)&=~\sup_{\nu\in\mathbb R}~\left[-f^\circ(\Lambda \nu)-g^\circ(-\nu)\right]\\
&=~\sup_{\nu\in\mathbb R}~\left[\min_{\zeta\ge 0}\Big\{c(\zeta)+xp\nu\zeta\Big\}+\min_{C(x,p)+\zeta\in [0,1]}\Big\{L(C+\zeta)+\nu\zeta\Big\}\right]\\
&=~\sup_{\nu\in\mathbb R}~\left[\min_{\zeta\ge 0}\Big\{c(\zeta)+xp\nu\zeta\Big\}+\min_{u\in [0,1]}\Big\{L(u)+\nu u\Big\}-C\nu\right]\\
&=~\sup_{\xi\in\mathbb R}~\left[\min_{\zeta\ge 0}\Big\{c(\zeta)-x\xi\zeta\Big\}+\min_{u\in [0,1]}\Big\{L(u)-u\cdot\dfrac{\xi}{p}\Big\}+\dfrac{C(x,p)}{p}\cdot\xi\right]\\
&=~\sup_{\xi\in\mathbb R}~H(x,\xi,p)=H^{\max}(x,p).
\end{align*}
Moreover, since $\displaystyle\sup_{\xi\in\mathbb R}H(x,\xi,p)$ is attained only at $\xi=\xi^\sharp(x,p)$ according to the strict concavity of $\xi\mapsto H(x,\xi,p)$, $(\hat u,\hat v)$
realizes the minimum in the right hand side of \eqref{eq:convH} if and only if
\[
\begin{cases}
f(\hat v)+f^\circ(\Lambda\xi^{\sharp}(x,p))-\Lambda\hat v\xi^\sharp(x,p)=0,\\
g(\Lambda \hat v)+g^\circ(-\xi^{\sharp}(x,p))+\Lambda \hat v \xi^\sharp(x,p)=0,
\end{cases}
\]
which implies $\hat v\ge 0$, $C(x,p)-px\hat v\in [0,1]$, and
\[
\begin{cases}
\displaystyle c(\hat v)+px\hat v\xi^\sharp(x,p)=\min_{\zeta\ge 0}~\{px\xi^\sharp(x,p)\zeta+c(\zeta)\},\\
\displaystyle L(C(x,p)-px\hat v)-px\hat v \xi^\sharp(x,p)=\min_{\nu\in\mathbb R}~\left\{\xi^{\sharp}(x,p)\nu+L(C(x,p)+\nu)\right\}.
\end{cases}
\]
The second relation can be rewritten as 
\[\displaystyle L(\hat u)+\hat u\xi^\sharp(x,p)=\min_{u\in[0,1]}\left\{\xi^{\sharp}(x,p)u+L(u)\right\}.\]
\end{proof}
Formula \eqref{eq:convH} allows us to give a simpler characterization of the minimum cost of a strategy keeping the debt-to-income ratio constant in time.
Indeed, given $x\in [0,x^*]$, we select $(u(x),v(x))$ keeping the debt-to-income ratio constant in time. This defines uniquely a value $p=p(x)$ by Definition 
\ref{def:const-strat} and impose a relation between $u(x)$ and $v(x)$. 
Then we take the minimum over all the costs of such strategies, i.e., the right hand side of formula \eqref{eq:convH}.
This naturally leads to the following definition.

\begin{definition}[Optimal cost for constant strategies]\label{def:optconst}\emph{
Given $x\in [0,x^*]$, we define
\[W(x)=\dfrac{1}{r}\cdot H^{\max}\left(x,p_c(x)\right),\]
where
\begin{equation}\label{eq:pcvc}
\begin{cases}p_c(x)=\dfrac{r+\lambda}{r+\lambda + v_c(x)},\\ \\
\displaystyle v_c(x)=\underset{v\ge 0}{\mathrm{argmin}}\left[L\left(\dfrac{(r+\lambda)(r-\mu)x}{r+\lambda+v}\right)+c(v)\right].\end{cases}
\end{equation}
For every $x\in [0,x^*]$, $W(x)$ denotes the minimum cost of a strategy keeping the DTI ratio constant in time.}
\end{definition}

The next results proves that if the debt-to-income ratio is sufficiently small, the optimal strategy keeping it constant does not use the devaluation of currency.

\begin{proposition}[Non-devaluating regime for optimal constant strategies]\label{prop:nondev}
Let $x_c\ge 0$ be the unique solution of the following equation in $x$ 
\[(r+\lambda)c'(0)=(r-\mu)xL'\left((r-\mu)x\right).\]
Then
\begin{itemize}
\item for all $x\in [0,\min\{x_c,x^*\}]$ we have $W(x)=\dfrac{1}{r}\cdot L((r-\mu)x)$ and $p_c(x)=1$,
\item for all $x\in ]\min\{x_c,x^*\},x^*]$ we have 
\begin{align*}
W(x)&=~\dfrac{1}{r}\left[L\left(\dfrac{(r+\lambda)(r-\mu)x}{r+\lambda+v_c(x)}\right)+c(v_c(x))\right],\\
p_c(x)&=~\dfrac{r+\lambda}{r+\lambda + v_c(x)}<1,
\end{align*}
where $v_c(x)>0$ solves the following equation in $v$ 
\[c'(v)=\dfrac{(r+\lambda)(r-\mu)x}{(r+\lambda+v)^2}\cdot L'\left(\dfrac{(r+\lambda)(r-\mu)x}{r+\lambda+v}\right).\]
\item for every $x\in]0,x^*[$ we have
\begin{equation}\label{eq:k-slop}
W'(x)=\dfrac{r-\mu}{r} p_c(x) L'(p_c(x)(r-\mu)x)<\xi^\sharp(x,p_c(x)).
\end{equation}
\end{itemize}
\end{proposition}
\begin{proof}
Given $x\in]0,x^*[$, we define the convex function 
\[F^x(v)\doteq \begin{cases}\dfrac{1}{r}\cdot \left[L\left(\dfrac{(r+\lambda)(r-\mu)x}{r+\lambda+v}\right)+c(v)\right],&\textrm{ if }v\ge 0,\,\dfrac{(r+\lambda)(r-\mu)x}{r+\lambda+v}\in [0,1],\\ \\
+\infty,&\textrm{otherwise}.\end{cases}\]
We compute
\[
\dfrac{d}{dv}F^x(v)=\dfrac{1}{r}\cdot \left[c'(v)-L'\left(\dfrac{(r+\lambda)(r-\mu)x}{r+\lambda+v}\right)\dfrac{(r+\lambda)(r-\mu)x}{(r+\lambda+v)^2}\right],
\]
which is monotone increasing and satisfies $\displaystyle\lim_{v\to+\infty}\dfrac{d}{dv}F^x(v)=+\infty$,
\[
\dfrac{d}{dv}F^x(v)~\ge~\dfrac{d}{dv}F^x(0)=\dfrac{1}{r}\cdot \left[c'(0)-L'\left((r-\mu)x\right)\dfrac{(r-\mu)x}{r+\lambda}\right]\,.
\]
Two cases may occur:
\begin{itemize}
\item If $\dfrac{d}{dv}F^x(0)\ge 0$, we have that $v=0$ realizes the minimum of $F$ on $[0,+\infty[$. This occours when $x\in [0,\min\{x_c,x^*\}]$ where 
$x_c$ is the unique solution of 
\[(r+\lambda)c'(0)=(r-\mu)xL'\left(\dfrac{(r+\lambda)(r-\mu)x}{r+\lambda}\right),\]
and it implies $W(x)=\dfrac{1}{r}\cdot L((r-\mu)x)$ and $p_c(x)=1$.\par\medskip\par
\item If we have $\min\{x_c,x^*\}<x\le x^*$, then there exists a unique point $v_c(x)>0$ such that $F'(v_c(x))=0$, and this point is characterized by 
\[
c'(v_c(x))=\dfrac{(r+\lambda)(r-\mu)x}{(r+\lambda+v_c(x))^2}\cdot L'\left(\dfrac{(r+\lambda)(r-\mu)x}{r+\lambda+v_c(x)}\right).
\]
The remaining statements follows noticing that for $\min\{x_c,x^*\}<x\le x^*$ we have
\begin{align*}
W'(x)~&=~\ \dfrac{\partial F^x}{\partial x}(v_c(x))+\dfrac{\partial}{\partial v}F^x(v_c(x)) \cdot v'_c(x)=\dfrac{\partial F^x}{\partial x}(v_c(x))\\
&=~\ \dfrac{r-\mu}{r} p_c(x) L'(p_c(x)(r-\mu)x),
\end{align*}
and deriving the explicit expression of $W(x)$ for $[0,\min\{x_c,x^*\}]$ yields the same formula.
Notice that, by \eqref{eq:Xisharp}, we have
\begin{align*}
\xi^\sharp(x,p_c(x))~&=~p_c(x) L'\left(\big[(\lambda+r)-(\lambda+\mu+v^{\sharp}(x,p_c(x)))p_c(x)\big]\cdot x\right)\\
&=~ p_c(x)L'\left(\left[(\lambda+r)-(\lambda+\mu+v^{\sharp}(x,p_c(x)))\cdot \dfrac{\lambda+r}{\lambda+r+v^{\sharp}(x,p_c(x))}\right]\cdot x\right)\\
&=~ p_c(x) L'(p_c(x)(r-\mu)\cdot x)>W'(x),
\end{align*}
where we used the fact that $L'$ is strictly increasing and, since the argument of $L'$ must be nonnegative, we have
\[\dfrac{\lambda+r}{\lambda+\mu+v^{\sharp}(x,p_c(x))}~\ge~p_c(x).\]
\end{itemize}
The proof is complete.
\end{proof}

\subsection{Existence of an equilibrium solution.} 
In the subsection, we will establish an existence result of a equilibrium solution to  the debt management problem \eqref{eq:dode1}-\eqref{eq:covec-det}. 
Before going to state our main theorem, we recall from Proposition \ref{prop:nondev} that $v_{c}$ is the unique solution to 
\[c'(v)=\dfrac{(r+\lambda)(r-\mu)x}{(r+\lambda+v)^2}\cdot L'\left(\dfrac{(r+\lambda)(r-\mu)x}{r+\lambda+v}\right)\,,\]
and 
\[p_c(x^*)=\dfrac{r+\lambda}{r+\lambda + v_c(x^*)} <1\,,\]
\begin{equation}\label{eq:w11}
W(x^*)=\dfrac{1}{r}\left[L\left(\dfrac{(r+\lambda)(r-\mu)x^*}{r+\lambda+v_c(x^*)}\right)+c(v_c(x^*))\right]\,.
\end{equation}

\begin{theorem}\label{thm:main2} 
Assume that the cost functions $L$ and $c$ satisfies the assumptions \textbf{(A1)-(A2)}, and moreover
\begin{equation}\label{eq:x-large}
W(x^*)>B\qquad\textrm{ and }\qquad \theta(x^*)\le p_c(x^*)\,.
\end{equation}
Then the debt management problem \eqref{eq:dode1}-\eqref{eq:covec-det} admits an equilibrium solution $(u^*,v^*,p^*)$ associated to  Lipschitz continuous value functions $V^*$ in feedback form such that $p^*$ is decreasing, $V^*$ is strictly increasing and
\[
V^*(x)~\leq~W^*(x)\qquad\forall x\in [0,x^*].
\]
\end{theorem}
Toward the proof of this theorem, we study  basic properties of the backward solutions of the  system of implicit ODEs \eqref{eq:Sdode}. In fact, an equilibrium solution will be constructed by a suitable concatenation of backward solutions.

\subsubsection{Backward solutions}  We first define the backward solution to the system \eqref{eq:Sdode} starting from $x^*$.
%
\begin{definition}[Backward solution for $x^*$]
Let $x\mapsto (Z(x,x^*),q(x,x^*))$ be the backward solution of the system of ODEs
\begin{equation}\label{eq:VB}
\begin{cases}Z'(x)&=~F^{-}(x,Z(x),q(x)),\\[4mm] q'(x)&=~G^{-}(x,Z(x),q(x)), \end{cases}
\qquad\textrm{with}\qquad\begin{cases}Z(x^*)&=~B\,,\\[4mm] q(x^*)&=~\theta(x^*).\end{cases}
\end{equation}
with $H_\xi(x,F^{-}(x,Z(x),q(x)),q(x))\ne 0$.
\end{definition}

\par\medskip\par

The following Lemma states some basic properties of the backward solution. In particular, the backward solution $Z(\cdot,x^*)$, starting from $B$ at $x^*$ with $W(x^*)<B$, 
survives backward at least  until the first intersection with the graph of $W(\cdot)$. Moreover, in this interval is monotone increasing and positive. In the same way, $q(\cdot,x^*)$ 
is always in $]0,1]$.

\begin{proposition}\label{prop:pro1}[Basic properties of the backward solution]
Set
\[
x^*_{W}:=~\begin{cases}0,&\textrm{ if }Z(x,x^*)<W(x)\textrm{ for all }x\in]0,x^*[,\\
\sup\{x\in ]0,x^*[:\, Z(x,x^*)\ge W(x)\},&\textrm{ otherwise }.\end{cases}
\]
Assume that 
\begin{equation}\label{eq:Asp1}
W(x^*)>B\qquad\textrm{ and }\qquad \theta(x^*)<\dfrac{r+\lambda}{r+\lambda+v^*(x^*,F^-(x^*,B,\theta(x^*))) }\,.
\end{equation}
Denote by $I_{x^*}\subseteq [0,x^*]$ the maximal domain of the backward equation \eqref{eq:VB}, 
define $y(x)$ to be the maximal solution of
\[\begin{cases}
\dfrac{dy}{dx}(x)=\dfrac{1}{H_{\xi}\left(x, Z'(x,x^*),q(x,x^*)\right)},\\ \\ y(x^*)=0,   
\end{cases}\]
and let $J_{x^*}$ the intersection of its domain with $[0,x^*]$. 
Then
\begin{enumerate}
\item $I_{x^*}\supseteq J_{x^*}\supseteq ]x^*_{W},x^*[$;
\item $Z(\cdot,x^*)$ is strictly monotone increasing in $]x^*_{W},x^*[$, and $Z(x,x^*)>0$ for all $x\in ]x^*_{W},x^*]$;
\item $q(x,x^*)\in ]0,1]$ for all $x\in ]x^*_{W},x^*]$. 
\end{enumerate}
\end{proposition}
\begin{proof}$ $\par
\textbf{1.} We first claim that $q(\cdot,x^*)$ is non-increasing on $J_{x^*}\bigcap ]x^*_{W},x^*[$ and thus
\begin{equation}\label{eq:de-p}
q'(x,x^*)=~\dfrac{[r+\lambda+v^*(x,Z'(x,x^*))]\cdot q(x,x^*)-(r+\lambda)}{H_{\xi}(x,Z'(x,x^*),q(x,x^*))}~\le~0, \textrm{ for all }x\in J_{x^*}\cap ]x^*_{W},x^*[\,.
\end{equation}
By contradiction, assume that there exists $x_1\in J_B\cap ]x_{BW},x^*[$  such that 
\begin{equation}\label{eq:as1}
q'(x_1,x^*)=\dfrac{[r+\lambda+v^*(x_1,Z'(x_1,x^*))]\cdot q(x_1,x^*)-(r+\lambda)}{H_{\xi}(x_1,Z'(x,x^*),q(x,x^*))}=0, \textrm{ and }q''(x_1,x^*)<0\,.
\end{equation}
This yields 
\[r+\lambda=[r+\lambda+v^*(x_1,Z'(x_1,x^*))]\cdot q(x_1,x^*)\textrm{ and } q(x_1,x^*)>0.\]
Two cases are considered:
\begin{itemize}
\item if $x_1Z'(x_1,x^*)\leq c'(0)$ then, recalling the monotonicity of $Z'(\cdot,x^*)$, we have that $xV'(x,x^*)\leq c'(0)$ for all $x\in J_{x^*}\cap ]x^*_{W},x^*[$ satisfying $x\le x_1$, and so
\[v^*(x,Z'(x,x^*))=0, \textrm{ for all }x\in J_{x^*}\cap ]x^*_{W},x^*~[\textrm{with}~x\le x_1.\]
Thus, $q(x_1,x^*)=1$ and 
\[
q'(x,x^*)=\dfrac{[r+\lambda]\cdot [q(x,x^*)-1]}{H_{\xi}(x,Z'(x,x^*),q(x,x^*))}\quad\forall x\in J_{x^*}\cap ]x^*_{W},x^*[~\textrm{with}~x\le x_1.
\]
This implies that  $q(x,x^*)=1$ for all $x\in J_{x^*}\cap ]x^*_{W},x^*[$ with $x\le x_1$. In particular, we have $q''(x_1,x^*)=0$, which yields a contradiction.
\item If $x_1Z'(x_1,x^*)>c'(0)$ then 
\[\dfrac{d}{dx}(v^*(x_1,Z'(x_1,x^*)))=\dfrac{Z''(x_1,x^*)x_1+Z'(x_1,x^*)}{c''(x_1Z'(x_1,x^*))}>0.\]
From the first equation of \eqref{eq:Sdode} and \eqref{eq:gradH}, it holds
\begin{align*}
rZ'(x_1,x^*)&=~H_x(x_1,Z',q)+H_{\xi}(x_1,Z',q)\cdot Z''(x_1,x^*)+H_{p}(x_1,Z,q)\cdot q'(x_1,x^*)\\
&=~\Big[(\lambda+r)-q(x_1,x^*)(\lambda +\mu +v^*(x,Z')\Big]\cdot \frac{Z'}{q}+H_{\xi}(x_1,Z',q)\cdot Z''(x_1,x^*)\\
&=~(r-\mu)\cdot Z'(x_1,x^*)+H_{\xi}(x_1,Z',q)\cdot Z''(x_1,x^*).
\end{align*}
Observe that  $Z'(x_1,x^*)>0$ and $H_{\xi}(x_1,Z'(x_1,x^*),q(x_1,x^*))>0$,  one obtains that 
\[
Z''(x_1,x^*)=\dfrac{\mu Z'(x_1,x^*)}{H_{\xi}(x_1,Z'(x_1,x^*),q(x_1,x^*))}>0\,.
\]
Taking the derivative respect to $x$ in both sides of the second equation of \eqref{eq:Sdode}, we have 
\begin{multline*}
\left [r+\lambda+(v^*(x,Z'(x,x^*))\right]\cdot q'(x,x^*)+q(x,x^*)\cdot\dfrac{d}{dx}v^*(x,Z'(x,x^*))\\
=q''(x,x^*)H_{\xi}(x,Z'(x,x^*),q(x,x^*))+q'(x,x^*)\dfrac{d}{dx}H_{\xi}(x,Z'(x,x^*),q(x,x^*))\,.
\end{multline*}
Recalling \eqref{eq:as1}, we obtain that 
\begin{equation}\label{eq:p''}
q''(x_1,x^*)=\dfrac{q(x_1,x^*)}{H_{\xi}(x_1,Z',q)}\cdot \dfrac{d}{dx}v^*(x_1,Z'(x_1,x^*))>0.
\end{equation}
and it yields a contradiction.
\end{itemize}
\par\medskip\par

Now assume that there exists $x_2\in J_{x^*}\cap ]x^*_{W},x^*[$ such that $H_{\xi}(x_2,Z'(x_2,x^*),q(x_2,x^*))=0$. One has that 
\[
\xi^{\sharp}(x_2,q(x_2,x^*))=Z'(x_2,x^*)\qquad\textrm{ and }\qquad Z(x_2,x^*)=\dfrac{1}{r}\cdot H^{\max}(x_2,q(x_2,x^*)).
\]
Moreover, 
\[
u^{\sharp}(x_2,q(x_2,x^*))=\big[(\lambda+r)-(\lambda+\mu+v^{\sharp}(x_2,q(x_2,x^*)))q(x_2,x^*)\big]\cdot x_2\,.
\]
On the other hand, since $q(x_2,x^*)\leq \dfrac{r+\lambda}{r+\lambda+v^{\sharp}(x,Z'(x_2,x^*))}$, we estimate
\begin{align*}
H^{\max}&(x_2,q(x_2,x^*))=L(u^{\sharp}(x_2,q(x_2,x^*)))+c(v^{\sharp}(x_2,q(x_2,x^*)))\\
&=~L\left(\big[(\lambda+r)-(\lambda+\mu+v^{\sharp}(x_2,q(x_2,x^*)))q(x_2,x^*)\big]\cdot x_2\right)+c(v^{\sharp}(x_2,q(x_2,x^*)))\\
&\geq~L\left(\dfrac{r+\lambda(r-\mu)x_2}{\lambda+\mu+v^{\sharp}(x_2,q(x_2,x^*))}\right)+c(v^{\sharp}(x_2,q(x_2,x^*)))\\
&\geq~H^{\max}(x_1,p_c(x_2)).
\end{align*}
Thus, 
\[
Z(x_2,x^*)=\dfrac{1}{r}\cdot H^{\max}(x_2,q(x_2,x^*))\geq\dfrac{1}{r}\cdot H^{\max}(x_2,p_c(x_2))=W(x_2),
\]
and it yields a contradiction.\par\medskip\par
\textbf{2.} By construction, $y(\cdot)$ is strictly monotone and invertible in $]x^*_{W},x^*]$, let $x=x(y)$ be its inverse, from the inverse function theorem we get
\[\begin{cases}
\dfrac{d}{dy}~Z(x(y),x^*)=Z'(x(y),x^*)\cdot H_{\xi}\left(x(y),Z'(x(y),x^*),q(x(y),x^*)\right),\\ \\
\dfrac{d}{dy}~q(x(y),x^*)=q'(x(y),x^*)\cdot H_{\xi}\left(x(y),Z'(x(y),x^*),q(x(y),x^*)\right).
\end{cases}\]
Since the map $\xi\mapsto H(x,\xi,q)$ is concave, it holds
\[H_{\xi}(x,0,q(x,x^*))~\ge~H_{\xi}(x,\xi,q(x,x^*))~\ge~H_{\xi}\left(x,Z'(x,x^*),q(x,x^*)\right),\]
for all $\xi\in \left[0,Z'(x,x^*)\right]$.
We have
\begin{align*}
r Z(x(y),x^*)&=~H\left(x(y), Z'(x(y),x^*),q(x(y),x^*)\right)\\
&\qquad\qquad=~\int_0^{Z'(x(y),x^*)}H_{\xi}(x,\xi,q(x(y),x^*))\,d\xi\\
&\ge~Z'(x(y),x^*)\cdot H_{\xi}(x,Z(x(y),x^*),q(x(y),x^*))=\dfrac{d}{dy}Z(x(y),x^*)
\end{align*}
and this yields $Z(x,x^*)\ge Be^{ry(x)}>0$ for all $x\in ]x^*_{W},x^*]$.
\par\medskip\par 

With a similar argument for $q(\cdot,x^*)$, we obtain
\[
(r+\lambda+v^*(x(y),Z'(x(y),x^*))\cdot q(x(y),x^*)-(r+\lambda)=\dfrac{d}{dy}q(x(y),x^*)),
\]
and so 
\[(r+\lambda)(q(x(y),x^*)-1)~\le~\dfrac{d}{dy}q(x(y),x^*)~\le~ (r+\lambda+v^*(x(y),Z'(x(y),x^*))\cdot q(x(y),x^*),
\]
which in particular implies that for all $x\in I_{x^*}\cap[0,x^*]$
\[
q(x,x^*)~\le ~1,\hspace{1cm} q(x,x^*)~\ge~\theta(x^*)\cdot e^{(r+\lambda+v^*(x,Z'(x,x^*))y(x)}>0,
\]
and so $q(x,x^*)\in ]0,1]$ for all $x\in ]x^*_{W},x^*]$.
\end{proof}
\par\medskip\par
As far as the graph of $Z(\cdot,x^*)$ intersects the graph of $W(\cdot)$, we have that $Z(\cdot,x^*)$ is no longer optimal.
We investigate now the local behavior of $Z(\cdot,x^*)$ and $W(\cdot)$ near to an intersection of their graphs.

\begin{lemma}[Comparison between optimal constant strategy and backward solution]\label{lemma:compW}
Let $I\subseteq ]0,x^*[$ be an open interval, $(Z,q):I\to [0,+\infty[\times]0,1[$ be a backward solution, and $\bar{x}\in \bar I$. 
Assume that 
\[
\lim_{\substack{x\to \bar x\\ x\in I}}Z(x)=W(\bar x).
\]
Then $p_c(\bar x)\ge \displaystyle\limsup_{\substack{x\to\bar x\\ x\in I}} q(x)$ and $W'(x)< F^{-}(x,W(x),p_c(x))$.
\end{lemma}
\begin{proof}
Let $\{x_j\}_{j\in\mathbb N}\subseteq I$ be a sequence converging to $\bar x$ and $q_{\bar x}\in [0,1]$ be such that $\displaystyle q_{\bar x}=\limsup_{x\to\bar x^+} q(x)=\lim_{j\to \infty}q(x_j)$.
By assumption, we have 
\[H^{\max}(x,p_c(x))=\lim_{j\to +\infty} H\left(x_j,Z'(x_j),q(x_j)\right)\le \lim_{j\to +\infty} H^{\max}(x_j,q(x_j))=H^{\max}(\bar x,q_{\bar x}).\]
Recalling Lemma \ref{lemma:nonincmax} (4), we have $p_c(\bar x)\ge q_{\bar x}$. By Proposition \ref{prop:nondev}, we have $W'(\bar x)<\xi^\sharp(\bar x,p_c(\bar x))$, and so
\[H(\bar x,W'(\bar x),p_c(\bar x))<H^{\max}(\bar x,p_c(\bar x))=rW(\bar x),\] 
thus, by applying the strictly increasing map $F^{-}(\bar x,\cdot,p_c(x))$ on both sides, we obtain $W'(x)< F^{-}(x,W(x),p_c(x))$.
\end{proof}
\par\medskip\par
\par
Since, the functions $F^-(x,Z,q)$ and $G^{-}(x,Z,q)$ are smooth for $H_{\xi}(x,Z,q)\neq 0$ but not only H\"{o}lder continuous with respect to $Z$ near to the surface
\[
\Sigma=\left\{(x,Z,q)~|~H_{\xi}(x,Z,q)=0\right\}\,.
\]
Thus, for any $x_0\in [0,x^*)$,  the definition of the solution of  the Cauchy problem 
\begin{equation}\label{eq:V-x0}
\begin{cases}Z'(x)&=~F^{-}(x,Z(x),q(x)),\\[4mm] q'(x)&=~G^{-}(x,Z(x),q(x)), \end{cases}
\qquad\textrm{with}\qquad\begin{cases}Z(x_0)&=~W(x_0)\,,\\[4mm] q(x_0)&=~p_c(x_0).\end{cases}
\end{equation}
requires some care. 
\par\medskip\par
\par
For any $\varepsilon>0$, we denote by $Z_{\varepsilon}(\cdot,x_0),q_{\varepsilon}(\cdot,x_0)$ the backward solution to \eqref{eq:V-x0} with the terminal data
\[
Z_{\varepsilon}(x_0,x_0)=W(y_0)-\varepsilon\qquad\textrm{ and }\qquad q_{\varepsilon}(x_0,x_0)=p_c(x_0).
\]
With the same argument in the proof of Proposition \ref{prop:pro1}, this solution is uniquely defined on a maximal interval $[a_{\varepsilon}(x_0),x_0]$ such that $Z_{\varepsilon}(\cdot,x_0)$ is increasing, $q_{\varepsilon}(\cdot,x_0)$ is decreasing and 
\[
Z_{\varepsilon}(a_{\varepsilon}(x_0),x_0)=W(a_{\varepsilon}(x_0)),\qquad \qquad q_{\varepsilon}(a_{\varepsilon}(x_0),x_0)\le p_{c}(a_{\varepsilon}(x_0)).
\]
Let $x^{\flat}$ be the unique solution to the equation
\begin{equation}\label{eq:x-flat}
c'(0)=x\cdot L'((r-\mu)x)\,.
\end{equation}
It is clear that $0<x^{\flat}<x_c$ where $x_c$ is defined in Proposition \ref{prop:nondev} as the unique solution to the equaiton
\[
(r+\lambda)c'(0)=(r-\mu)xL'\left((r-\mu)x\right).
\]
Two cases are considered:
\begin{itemize}
\item \textbf{CASE 1:} For any $x_0\in ]0,x^{\flat}]$, we claim that  
\[
a_{\varepsilon}(x_0)=0,\qquad q_{\varepsilon}(x,x_0)=1\qquad\forall x\in [0,x_0]\,,
\]
and $Z_{\varepsilon}(\cdot,x_0)$ solves backward the following ODE  
\begin{equation}\label{eq:odep=1}
Z'(x)=F^-(x,Z(x),1),\qquad Z(x_0)=W(x_0)-\varepsilon
\end{equation}
for $\varepsilon>0$ sufficiently small. Indeed, let $Z_1$ be the unique backward solution of \eqref{eq:odep=1}. From \eqref{eq:Xisharp}, it holds
\[
F^-(x,W(x),1)=\xi^{\sharp}(x,1)=L'((r-\mu)x)>\dfrac{r-\mu}{r}\cdot L'((r-\mu)x)=W'(x)
\]
for all $x\in ]0,x^{\flat}]$. As in \cite{BMNP}, a contradiction argument yields
\[
0<Z_1(x)<W(x)\qquad\forall x\in ]0,x_0]\,.
\]
Thus, $Z_1$ is well-defined on $[0,x_0]$ and $Z_1(0)=0$. On the other hand, it holds
\[
Z'(x_1)=F^-(x,Z(x),1)\le \xi^{\sharp}(x,1)=L'((r-\mu)x)\le L'((r-\mu)x^{\flat})
\]
for all $x\leq x^{\flat}$ and \eqref{eq:x-flat} implies that 
\[
v^*(x,Z'_1(x))=0\qquad\forall x\in [0,x^{\flat}]\,.
\]
Therefore, $(Z_1(x),1)$ solves \eqref{eq:V-x0} and the uniqueness yields 
\[
Z_{\varepsilon}(x,x_0)=Z_1(x)\qquad\textrm{ and }\qquad q_{\varepsilon}(x,x_0)=1\qquad\forall x\in [0,x_0]\,.
\] 
Thanks to the monotone increasing property of the map $\xi\to F^-(x,\xi,1)$, a pair $(Z(\cdot,x_0),q(\cdot,x_0))$ denoted by 
\[
q(x,x_0)=1\qquad\textrm{ and }\qquad Z(x,x_0)=\sup_{\varepsilon>0}Z_{\varepsilon}(x,x_0)\qquad\forall x\in [0,x_0]
\]
is the unique solution of \eqref{eq:V-x0}. If the initial size of the debt is $\bar{x}\in [0,x_0]$ we think of $Z(\bar{x},x_0)$ is as the expected cost of 
\eqref{eq:cost-feed}-\eqref{eq:cont-feed} with $p(\cdot,x_0)=1$, $x(0)=x_0$ achieved by the feedback strategies
\begin{equation}\label{eq:u-1xx}
u(x,x_0)=\underset{w\in [0,1]}{\mathrm{argmin}}\left\{ L(w) - \, Z'(x,x_0)\cdot w\right\},\quad v(x,x_0)=0
\end{equation}
for all $x\in [0,x_0]$. With this strategy, the debt has the asymptotic behavior $x(t)\to x_0$ as $t\to\infty$.
\par\medskip\par
\item \textbf{CASE 2:} For $x_0\in (x^{\flat},x_W^*]$,  system of ODEs \eqref{eq:V-x0} does not admit a unique solution in general since it is not monotone, it. 
The following lemma will provide the existence result of \eqref{eq:V-x0} for all $x_0\in (x^{\flat},x^*_{W}]$.
\begin{lemma}\label{lemma:bw} There exists a constant $\delta_{\flat}>0$ depending only on $x^{\flat}$ such that  for any $x_0\in \left(x^{\flat},x^*_W\right)$, it holds  
\[
x_0-a_{\varepsilon}(x_0)\ge \delta_{x^\flat}\qquad\forall \varepsilon\in (0,\varepsilon_0)
\]
for some $\varepsilon_0>0$ sufficiently small.
\end{lemma}
\begin{proof} From \eqref{eq:k-slop} and \eqref{eq:Xisharp}, it holds
\[
\inf_{x\in [x^{\flat},x_W^*]}\left\{\xi^{\sharp}(x,p_c(x))-W'(x)\right\}=\delta_{1,\flat}>0.
\]
In particular, we have 
\[
F^{-}(x_0,W(x_0),p_{c}(x_0))-W'(x_0)=\delta_{1,\flat}.
\]
By continuity of the map $\eta\mapsto F^-(x_0,\eta,p_c(x_0))$ on $[0,W(x)]$, one can find a constant $\varepsilon_1>0$ sufficiently small such that 
\[
F^-(x_0,\eta,p_c(x_0))\ge W'(x_0)+\dfrac{\delta_{1,\flat}}{2}\quad\forall \xi\in [W(x_0)-\varepsilon_1,W(x_0)].
\]
On the other hand, the continuity of $W'$ yields
\[
\delta_{2,\flat}=\sup\left\{s\geq 0~\Big|~W'(x_0-\tau)<W'(x_0)+\dfrac{\delta_{1,\flat}}{4}\quad\forall \tau \in [0,s]\right\}>0.
\] 
For a fixed $\varepsilon\in (0,\varepsilon_1)$, denote by 
\[
x_1\doteq \inf\left\{s\in (0,x_0]~\Big|~F^{-}\big(x,Z_{\varepsilon}(x,x_0),q_{\varepsilon}(x,x_0)\big)>W'(x)~\forall x\in (s,x_0]\right\}.
\]
If $x_1>x_0-\delta_{2,\bar{x}}$ then it holds
\begin{equation}\label{eq:cd1}
F^{-}\big(x_1,Z_{\varepsilon}(x_1,x_0),q_{\varepsilon}(x_1,x_0)\big)=W'(x_1)\le W'(x_0)+\dfrac{\delta_{1,\flat}}{4}
\end{equation}
and there exists $x_2\in (x_1,x_0]$ such that 
\begin{equation}\label{eq:cd2}
F^{-}\big(x_2,Z_{\varepsilon}(x_2,x_0),q_{\varepsilon}(x_2,x_0)\big)=W'(x_0)+\dfrac{\delta_{1,\flat}}{2}
\end{equation}
and 
\begin{equation}\label{eq:k-cod}
F^{-}\left(x,Z_{\varepsilon}(x,x_0),q_{\varepsilon}(x,x_0)\right)\le W'(x_0)+\dfrac{\delta_{1,\flat}}{2}\qquad\forall x\in [x_1,x_2].
\end{equation}
Recalling that $(x,\eta,p)\mapsto F^-(x,\eta,p)$ is defined by $H(x,F^-(x,\eta,p),p)=r\eta$, by the implicit function theorem, set $\xi=F^-(x,\eta,p)$,
we have
\begin{align*}
\dfrac{\partial}{\partial p}F^-(x,\eta,p)&=~-\dfrac{H_p(x,\xi,p)}{H_\xi(x,\xi,p)}\\
&=~\dfrac{\xi}{p}\cdot \dfrac{u^*(x,\xi,p)-x (\lambda +r)}{u^*(x,\xi,p)-x (\lambda +r)+xp(\lambda +\mu +v^*(x,\xi))}\\
&=~\left(1+\dfrac{x(\lambda+\mu+v^*(x,\xi)}{H_\xi(x,\xi,p)}\right)\dfrac{\xi}{p}>\dfrac{F^-(x,\eta,p)}{p}>0.
\end{align*}
Since $q_{\varepsilon}(\cdot,x_0)$ is decreasing, it holds
\[
F^{-}\big(x_1,Z_{\varepsilon}(x_1,x_0),q_{\varepsilon}(x_1,x_0)\big)\ge F^{-}\big(x_1,Z_{\varepsilon}(x_1,x_0),q_{\varepsilon}(x_2,x_0)\big),
\]
and \eqref{eq:cd1}-\eqref{eq:cd2} yield
\[
F^{-}\big(x_2,Z_{\varepsilon}(x_2,x_0),q_{\varepsilon}(x_2,x_0)\big)-F^{-}\big(x_1,Z_{\varepsilon}(x_1,x_0),q_{\varepsilon}(x_2,x_0)\big)\ge \dfrac{\delta_{1,\flat}}{4}.
\]
On the other hand, from \eqref{eq:gradH} one shows that  the map $x\to F^-(x,\eta,p)$ is monotone decreasing and thus 
\begin{equation}\label{eq:F-11}
F^{-}\big(x_2,Z_{\varepsilon}(x_2,x_0),q_{\varepsilon}(x_2,x_0)\big)-F^{-}\big(x_2,Z_{\varepsilon}(x_1,x_0),q_{\varepsilon}(x_2,x_0)\big)\ge \dfrac{\delta_{1,\flat}}{4}.
\end{equation}
Observe that the map $\eta\to F^-(x,\eta,p)$ is H\"{o}lder continuous due to Lemma \ref{lemma:holder}. More precisely, there exist a constant $C_{x^{\flat}}>0$ such that  
\[\left|F^-(x,\eta_2,p)-F^-(x,\eta_1,p)\right|\le C_{x^{\flat}}\cdot \big|\eta_2-\eta_1\big|^{\frac{1}{2}}\]
for all $\eta_1,\eta_2\in (0,W(x)]$, $x\in[\bar{x},x^*]$, $p\in [\theta(x^*),1]$. From \eqref{eq:F-11} it holds
\[
\displaystyle\left|Z_{\varepsilon}(x_2,x_0)-Z_{\varepsilon}(x_1,x_0)\right|\ge \dfrac{\delta^2_{1,\flat}}{16C^2_{x^{\flat}}}.
\]
Recalling \eqref{eq:k-cod}, we have 
\[
Z'_{\varepsilon}(x,x_0)=F^{-}\left(x,Z_{\varepsilon}(x,x_0),q_{\varepsilon}(x,x_0)\right)\le W'(x_0)+\dfrac{\delta_{1,x^{\flat}}}{2}\qquad\forall x\in [x_1,x_2]
\]
and it yields
\[
|x_2-x_1|\ge \displaystyle \dfrac{\delta^2_{1,x^{\flat}}}{8C^2_{x^{\flat}}[2W'(x_0)+\delta_{1,x^{\flat}}]}.
\]
Therefore, 
\[
x_0-a_{\varepsilon}(x_0)\ge \delta_{x^\flat}\doteq \min~\left\{\delta_{1,x^{\flat}},\dfrac{\delta^2_{1,x^{\flat}}}{8C^2_{x^{\flat}}[2W'(x_0)+\delta_{1,x^{\flat}}]} \right\}>0.
\]
\end{proof}
\begin{remark} 
In general, the backward Cauchy problem \eqref{eq:V-x0} may admits more than one solution.
\end{remark}
As a consequence of Lemma \ref{lemma:bw} , there exists a sequence $\{\varepsilon_{n}\}_{n\geq 0}\to 0+$ 
such that  $\{(Z_{\varepsilon_n}(\cdot,x_0),q_{\varepsilon_n}(\cdot,x_0))\}_{n\geq 1}$ converges to  $(Z(\cdot,x_0),q(\cdot,x_0))$ 
which is a solution of \eqref{eq:V-x0}. With the same argument in the proof of Proposition \ref{prop:pro1}, we can extend backward the solution $(Z(\cdot,x_0),q(\cdot,x_0))$ until $a(x_0)$ such that 
\[\lim_{x\to a(x_0)+}~Z(a(x_0),x_0)=W(a(x_0)),\]
and Lemma \ref{lemma:compW} yields $ \lim_{x\to a(x_0)+}q(a(x_0),x_0)\leq p_c(a(x_0))$. If the initial size of the debt is $\bar{x}\in [a(x_0),x_0]$ 
we think of $Z(\bar{x},x_0)$ is as the expected cost of \eqref{eq:cost-feed}-\eqref{eq:cont-feed} with $p(\cdot,x_0)$, $x(0)=x_0$ achieved by the feedback strategies
\begin{equation}\label{eq:u-xx}
u(x,x_0)=\underset{w\in [0,1]}{\mathrm{argmin}}\left\{ L(w) - \,\dfrac{Z'(x,x_0)}{p(x,x_0)}\cdot w\right\},\quad v(x,x_0)=\underset{v\geq 0}{\mathrm{argmin}}\Big\{c(v)-vxZ'(x,x_0)\Big\}\,.
\end{equation}
With this strategy, the debt has the asymptotic behavior $x(t)\to x_0$ as $t\to\infty$.\par\medskip\par
\end{itemize}

\begin{figure}[ht]
\begin{tikzpicture}[line cap=round,line join=round,>=triangle 45,x=0.1 \textwidth,y= 0.1 \textwidth]
\draw[->,color=black] (-0.5,0) -- (8,0);
\draw[->,color=black] (0.,-0.5) -- (0,5);
\clip(-0.5,-0.5) rectangle (8,5);
\draw[color=red,smooth,samples=100,domain=0:5] plot(\x,{6.0/(6.0-(\x))-1.0});
\draw[color=blue,smooth,samples=100,domain=0:2.17] plot(\x,{(\x)^(3.0)/18.0});
\draw[color=blue,smooth,samples=100,domain=2.17:3.44] plot(\x,{((\x)-2.17)^(4.0)/4.0+((\x)-2.17)/10.0+0.57});
\draw[color=blue,smooth,samples=100,domain=3.44:4.44] plot(\x,{((\x)-3.44)^(4.0)+((\x)-3.44)/2.0+1.34});
\draw[color=blue,smooth,samples=100,domain=4.44:7] plot(\x,{((\x) - 4.44)^2 / 20 + ((\x) - 4.44) / 8 + 2.85});

\draw[dotted] (2.17,0.57)--(2.17,0);
\draw[dotted] (2.80,0.875)--(2.80,0);
\draw[dotted] (3.44,1.34)--(3.44,0);
\draw[dotted] (4.44,2.85)--(4.44,0);
\draw[dotted] (7,3.5)--(7,0);
\draw[dotted] (7,3.5)--(0,3.5);

\draw[->] (0,0)--(1.08,0);\draw[->] (2.17,0)--(2.485,0);
\draw[->] (2.80,0)--(3.12,0);\draw[->] (3.44,0)--(3.94,0);
\draw[->] (4.44,0)--(5.72,0);

\node[color=black, anchor=north east] at (0,0) {$O$};

\draw[fill=blue] (2.17,0.57) circle (2.0pt);
\draw[fill=blue] (2.17,0) circle (2.0pt);
\node[color=blue, anchor = north] at (2.17,0) {$x_3^{\phantom{\flat}}$};

\draw[fill=blue] (3.44,1.34) circle (2.0pt);
\draw[fill=blue] (3.44,0) circle (2.0pt);
\node[color=blue, anchor = north] at (3.44,0) {$x_2^{\phantom{\flat}}$};

\draw[fill=blue] (4.44,2.85) circle (2.0pt);
\draw[fill=blue] (4.44,0) circle (2.0pt);
\node[color=blue, anchor = north] at (4.44,0) {$x_1^{\phantom{\flat}}$};

\draw[fill=blue] (7,3.5) circle (2.5pt);
\draw[fill=blue] (7,0) circle (2.5pt);
\draw[fill=blue] (0,3.5) circle (2.5pt);
\node[color=blue, anchor = north] at (7,0) {$x^{*\phantom{\flat}}$};
\node[color=blue, anchor = east] at (0,3.5) {$B$};

\draw[fill=black] (2.8,0.875) circle (2.5pt);
\draw[fill=black] (2.8,0) circle (2.5pt);
\node[color=black, anchor = north] at (2.8,0) {$x^{\flat}$};

\node[color=red, anchor=east] at (4.8,4) {$W(x)$};
\node[color=blue] at (5.72,2.85) {$Z(x)$};
\node[color=blue] at (3.94,1.30) {$Z_1(x)$};
\node[color=blue] at (3,0.50) {$Z_2(x)$};
\node[color=blue, anchor=north] at (1.08,0) {$Z_3(x)$};
\end{tikzpicture}
\label{fig:construction}\caption{Construction of a solution: starting from $(x^*,B)$ we solve backward the system until the first touch with the graph of $W$ at $(x_1,W(x_1))$.
Then we restart by solving backward the system with the new terminal conditions $(W(x_1),p_c(x_1))$, until the next touch with the graph of $W$ at $(x_2,W(x_2))$ and so on.
In a finite number of steps we reach the origin. If a touch occurs at $x_{n_0}<x^\flat$ then the backward solution from $x_{n_0}$ reaches the origin with $q\equiv 1$.
Given an initial value $\bar x$ of the DTI, if $0\le x_{n+1}<\bar x<x_n<x_1$ the the optimal strategy let the DTI increase asymptotically to $x_n$ (no banktuptcy), while
if $x_1<\bar x<x^*$ then the optimal strategy let the DTI increase to $x^*$, thus providing bankruptcy in finite time.}
\end{figure}
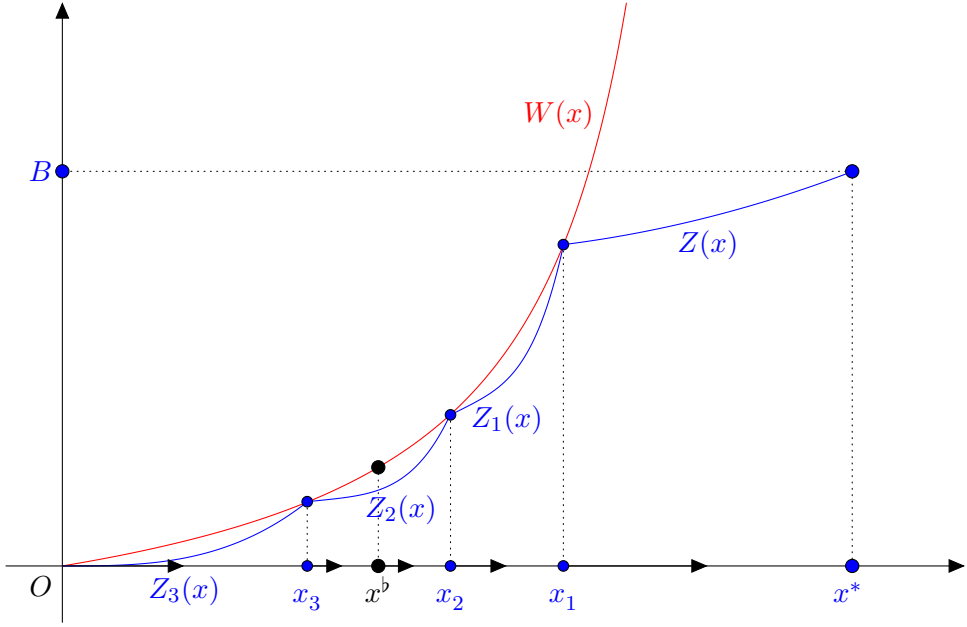

\subsubsection{Construction of an equilibrium solution.} 
We are now ready to construct an solution tothe system of Hamilton-Jacobi equation \eqref{eq:Sdode}. By induction, we define a family of back solutions as follows:
\[
x_1\doteq x^*_{W}, \qquad (Z_1(x),q_1(x))=\left(Z(x,x^*),q(x,x^*)\right)\qquad\forall x\in [x_1,x^*]
\]
and 
\[
x_{n+1}\doteq a(x_n),\qquad (Z(x,x_n),q(x,x_n))\qquad\forall x\in [x_{n+1},x_{n}]\,.
\]
From Case 1 and Lemma \ref{lemma:bw}, there exists a natural number $N_0<1+\dfrac{x^*-x^{\flat}}{\delta_{x^{\flat}}}$ 
such that our construction will be stop in $N_0$ step, i.e.,
\[
x_{N_0}>0,\qquad a(x_{N_0})=0\qquad\textrm{ and }\qquad \lim_{x\to a(x_{N_0})}Z(x,x_{N_0})=0\,.
\]
We will show that a feedback equilibrium solution to the debt management problem is obtained
as follows
\begin{equation}\label{eq:sol}
\left(V^*(x),p^*(x)\right)=\begin{cases}\left(Z(x,x^*),q(x,x^*)\right)\qquad\forall x\in (x_W,x^*],\\[4mm] 
\left(Z(x,x_{k}),q(x,x_k)\right) \qquad\forall x\in (a(x_k),x_k], k\in \{1,2,\dots, N_0\}\,.\end{cases}\\[4mm]\,.
\end{equation}
\begin{equation}\label{eq:sol-u}
u^*(x)=\underset{w\in [0,1]}{\mathrm{argmin}}\left\{ L(w) - \dfrac{(V^*)'(x)}{p^*(x)}\cdot w\right\},
\end{equation}
\begin{equation}\label{eq:sol-v}
v^*(x)=\underset{v\geq 0}{\mathrm{argmin}}\left\{c(v)-vx(V^*)'(x)\right\}\,.
\end{equation}
\begin{proof}[Proof of Theorem \ref{thm:main2}]
From the monotone increasing property  of the maps  $\xi\mapsto v^*(x^*,\xi)$, $\eta\mapsto F^-(x^*,\eta,\theta(x^*))$ and $p\mapsto F^-(x^*,W(x^*), p)$, we have 
\begin{multline*}
\theta(x^*)\cdot (r+\lambda+v^*(x^*,F^-(x^*,B,\theta(x^*)))\\
<p_c(x^*)\cdot (r+\lambda+v^*(x^*,F^-(x^*,W(x^*),p_c(x^*)))=r+\lambda
\end{multline*}
and it yields \eqref{eq:Asp1}. By Proposition \ref{prop:pro1} and Lemma \ref{lemma:bw}, a pair $V^*(\cdot),p^*(\cdot)$ in \eqref{eq:sol} is well-defined on $[0,x^*]$. In the remaining steps,  we show that $V^*, p^*, u^*,v^*$ provide an equilibrium solution. Namely, they satisfy the properties (i)-(ii) 
in Definition \ref{def:eqsol}.\par\bigskip\par
\textbf{1.} To prove (i), let $V(\cdot)$ be the value function 
for the optimal control problem \eqref{eq:cost-feed}-\eqref{eq:cont-feed}. For any initial value,  $x(0)=x_0\in [0,x^*]$,  
the feedback control $u^*$ and $v^*$ in  \eqref{eq:sol-u}-\eqref{eq:sol-v}
yields the cost $V^*(x_0)$. This implies
\[V(x_0)~\leq ~V^*(x_0).\]  
To prove the converse inequality we need to show that, for any measurable control $u:[0,+\infty[\,\mapsto [0,1]$ and $v:[0,+\infty[\to [0,+\infty[$, calling $t\mapsto x(t)$ the solution to 
\begin{equation}\label{eq:ode7}
\dot{x}(t)=\left(\dfrac{\lambda+r}{p^*(x(t))} -\lambda -\mu-v(t)\right) x(t) - \dfrac{u(t)}{p^*(x(t))}\,,\qquad\qquad x(0)=x_0\,,
\end{equation}
one has
\begin{equation}\label{eq:upVx_1}
\int_0^{T_b} e^{-rt}[L(u(x(t)))+c(v(x(t)))]\, dt + e^{-r T_b} B\ge V^*(x_0),
\end{equation}
where 
\[T_b=\inf\, \bigl\{t\geq 0\,;~~x(t)=x^*\bigr\}\] is the bankruptcy time
(possibly with $T_b=+\infty$).
\par\bigskip\par 
For $t\in [0, T_b]$, consider the
absolutely continuous function 
\[\phi^{u,v}(t)\doteq  \int_{0}^te^{-rs}\cdot [L(u(s))+c(v(s))]~ds + e^{-rt}V^*
(x(t)).\]
At any Lebesgue point $t$ of $u(\cdot)$ and $v(\cdot)$, recalling that $(V^*,p^*)$ solves the system \eqref{eq:Sdode}, we compute
\begin{align*}
&\displaystyle \frac{d}{dt}\phi^{u,v}(t)=e^{-rt}\cdot \Big[L(u(t))+c(v(t))-rV^*(x(t))+
(V^*)'(x(t))\cdot\dot{x}(t)\Big]\\[4mm]
&=\displaystyle ~e^{-rt}\cdot\Big[L(u(t))+c(v(t))-rV^*(x(t
))\\[4mm]
&\qquad\qquad\qquad+(V^*)'(x(t))\left(\left(\frac{\lambda+r}{p^*(x(t))}-\lambda-\mu-v(t)\right)x(t)-\frac{u(t)}{p^*(x(t))}\right)\Big]\\[4mm]
&\displaystyle\ge~ e^{-rt}\cdot\Big[ \min_{\omega\in[0,1]}
\left\{L(\omega)-\frac{(V^*)'(x(t))}{p^*(x(t))}\,\omega\right\} +\min_{\zeta\in [0+\infty[}\left\{c(\zeta)-(V^*)'(x(t))x(t)\,\zeta\right\}
\\[4mm]
&\qquad\qquad\qquad\qquad\qquad\qquad\qquad+ \left(\frac{\lambda+r}{p^*(x(t))}-\lambda-\mu\right)
x(t)(V^*)'(x(t))-rV^*(x(t))\Big]\\[3mm]
&= ~e^{-rt}\cdot \Big[H\left(x(t),(V^*)'(x(t)), p^*(x(t))\right)-rV^*(x(t))\Big]= 0.
\end{align*}
Therefore,
\[V^*(x_0)=\phi^{u,v}(0)\le  \lim_{t
\rightarrow T_b-
}\phi^{u,v}(t
)=\int_{0}^{T_b
} e^{-rt}\cdot [L(u(t))+c(v(t))]~dt + e^{-rT_b
}B,\]
proving \eqref{eq:upVx_1}.
\par\bigskip\par
\textbf{2.} It remains to check (ii).  The case $x_0=0$ is trivial.
Two main cases will be considered.
\par\bigskip\par
\emph{CASE 1:  $x_0\in ]x_1,x^*]$.} Then $x(t)>x_1$ for all 
$t\in [0,T_b]$. This implies
\[
\dot{x}(t)=H_{\xi}(x(t),Z(x(t),x^*),q(x(t),x^*))\,.
\]
From the second equation in  \eqref{eq:Sdode} it follows 
\[\dfrac{d}{dt}~p(x(t))=p'(x(t))\dot x(t)=(r+\lambda+v^*(x(t)))p(x(t))-(r+\lambda),\]
Therefore, for every $t\in [0, T_b]$ one has
\[
p(x(0))=p(t)\cdot\int_{0}^{t} e^{-(r+\lambda+v^*(x(\tau)))}~d\tau+\int_{0}^t (r+\lambda)\int_{0}^{\tau} e^{-(r+\lambda+v^*(x(s)))}~ds~d\tau
\]
Letting $t\to T_b$ we obtain
\[
p(x_0)=\int_{0}^{T_b} (r+\lambda)\int_{0}^{\tau} e^{-(r+\lambda+v^*(x(s)))}~ds~d\tau+\theta(x^*)\cdot\int_{0}^{T_b} e^{-(r+\lambda+v^*(x(\tau)))}~d\tau
\]
proving (ii).
\par\bigskip\par
 \emph{CASE 2: $x_0\in \, [a(x_k),x_{k}[$ for $k\in\{1,2,...,N_0\}$.}. In this case,  $T_b=+\infty$ and $x(t)\in [a_{x_k},x_k[$ such that 
\[
\lim_{t\to +\infty}~x(t)=x_{k}\,.
\]
With a similar computation, one has 
\[
p(x_0)=\theta(x^*)\cdot\int_{0}^{\infty} e^{-(r+\lambda+v^*(x(\tau)))}~d\tau
\]
proving (ii).
\end{proof}
\par\bigskip\par
\subsection{Dependence on $x^*$} As in the stochastic case, we now study the behavior of the total cost for servicing when the maximum size $x^*$ of the DTI, at which bankruptcy is declared, becomes very large. 
\begin{proposition} Let $(V(x,x^*),p(x,x^*))$ be constructed in Theorem \ref{thm:main2}. The following holds:
\begin{itemize}
\item [(i)] if $\displaystyle\limsup_{s\to +\infty}\theta(s)s=R<+\infty$  then 
\begin{equation}\label{eq:bb}
\liminf_{x^*\to+\infty}V(x,x^*)\ge B\cdot \left(1-\frac{R}{x}\right)^{\frac{r}{r+\lambda}}
\end{equation}
for all 
\[x\ge \frac{1}{r-\mu}\cdot \max\left\{4,\frac{4B}{L'(0)}, \frac{4C_1B}{c'(0)}, 2C_1c^{-1}(rB)\right\}.\]
\item [(ii)] if $\displaystyle\lim_{s\to+\infty}~\theta(s)s=+\infty$ then 
\begin{equation}\label{eq:vvD}
\limsup_{x^*\to\infty}~V(x,x^*)=0\qquad\forall x\in [0,x^*[
\end{equation}
\end{itemize}
\end{proposition}
\begin{proof} \textbf{1.} We first provide an upper bound on $v(\cdot,x^*)$. For all $x\geq\frac{4}{r-\mu}$, from \eqref{eq:Sdode} and \eqref{eq:detHamil}, we estimate 
\begin{align*}
H(x,\xi,p)&\geq~\min_{v\geq 0}\left\{c(v)-x\xi v\right\}+[(r-\mu)x-1]\cdot \frac{\xi}{p}\\
&\geq~\min_{v\geq 0}\left\{c(v)-x\xi v\right\}+	\frac{(r-\mu)x}{2}\cdot \frac{\xi}{p}\doteq K(x,\xi,p)
\end{align*}
for all $\xi,p >0$ and $x\geq \dfrac{2}{r-\mu}$. One computes 
\[K_{\xi}(x,\xi,p)=\frac{(r-\mu)x}{2p}-xv_K\]
where 
\[v_{K}=\begin{cases}0,&\textrm{ if }0\le x\xi< c'(0),\\ (c')^{-1}(x\xi),&\textrm{ if }x\xi\ge c'(0)>0.\end{cases}\]
This implies that the maximum of $K$ is achieved for $v_K=\frac{r-\mu}{2p}$ and its value is 
\[
\max_{\xi\geq 0}~K(x,\xi,p)=K(x,\xi_K,p)=c\left(\frac{r-\mu}{2p}\right)\qquad\textrm{ with }\qquad \xi_K=\frac{c'(v_K)}{x}\,.
\]
Thus, the monotone increasing property of the map $\xi\to H(x,\xi,p(x,x^*))$ on the interval $\big[0,\xi^{\sharp}(x,p(x,x^*))\big]$ implies that 
\begin{equation}\label{eq:v-cd}
F^-(x,V(x,x^*),p(x,x^*))<\xi_K\qquad\mathrm{\Longrightarrow}\qquad v(x,x^*)\le \frac{r-\mu}{2p(x,x^*)}\,.
\end{equation}
provided that $c\left(\frac{r-\mu}{2p(x,x^*)}\right)\geq rB$. From \eqref{eq:Sdode}) and \eqref{eq:detHamil}, one has 
\begin{align*}
rB\geq&-x V'(x,x^*)v(x,x^*)+[(r-\mu)x-u(x,x^*)]\cdot \frac{V'(x,x^*)}{p(x,x^*)}\\
\geq& \left[\frac{(r-\mu)x}{2}-1\right]\cdot \frac{V'(x,x^*)}{p(x,x^*)}\ge \frac{(r-\mu)x}{4}\cdot \frac{V'(x,x^*)}{p(x,x^*)}\,.
\end{align*}
Thus,  if 
\begin{equation}\label{eq:cod11}
p(x,x^*)\le  \min\left\{\frac{r-\mu}{2c^{-1}(rB)}, \frac{(r-\mu)c'(0)}{4B}\right\}\quad\textrm{ and }\quad x\ge  \max\left\{\frac{4}{r-\mu},\frac{4B}{(r-\mu)L'(0)}\right\}
\end{equation}
then 
\begin{equation}\label{eq:u=0}
\frac{V'(x,x^*)}{p(x,x^*)}\le  \frac{4B}{(r-\mu )x}\le L'(0)\qquad\mathrm{\Longrightarrow}\qquad u(x,x^*)=0\,.
\end{equation}
and 
\begin{equation}\label{eq:v=0-D}
V'(x,x^*)x\le \frac{4B}{r-\mu}\cdot p(x,x^*)\le c'(0)\qquad\mathrm{\Longrightarrow}\qquad v(x,x^*)=0\,.
\end{equation}
In this case, from \eqref{eq:Sdode}, \eqref{eq:detHamil} and \eqref{eq:gradH}, it holds
\[
(r+\lambda)(p(x,x^*)-1)=\left(\dfrac{\lambda+r}{p(x,x^*)}-\lambda-\mu\right)xp'(x,x^*)\,.
\]
Thus, 
\[
p(x,x^*)=\dfrac{\theta(x^*) x^*}{x}\cdot \left(\dfrac{1-p(x,x^*)}{1-\theta(x^*)}\right)^{\frac{r-\mu}{r+\lambda}}
\]
provided that \eqref{eq:cod11} holds.
\par\medskip\par
\textbf{2.} Assume that 
\[
\limsup_{s\in [0,+\infty)}~\theta(s)s=R<+\infty\,,
\]
we have 
\[
\sup_{s\in [0,+\infty)}~\theta(s)s=C_1
\]
 for some $C_1<+\infty$. Since $p(\cdot,x^*)$ is increasing, it holds
\begin{equation}\label{eq:upb-p}
p(x,x^*)\le \dfrac{\theta(x^*) x^*}{x}\le \frac{C_1}{x}\qquad\textrm{if \eqref{eq:cod11} holds}\,.
\end{equation}
Denote by 
\[M\doteq \dfrac{1}{r-\mu}\cdot \max\left\{4,\frac{4B}{L'(0)}, \frac{4C_1B}{c'(0)}, 2C_1c^{-1}(rB)\right\}\,,\]
we then have
\[u(x,x^*)=v(x,x^*)=0\qquad\forall x\in [M,x^*], x^*\geq M\,.\]
Recalling \eqref{eq:Sdode}, \eqref{eq:detHamil} and \eqref{eq:gradH}, we have 
\begin{equation}\label{eq:ode---}
\begin{cases}
V'(x,x^*)&=~\dfrac{rp}{[(\lambda+r)-(\lambda+\mu)p(x,x^*)]x}\, V \,,\\[6mm]
p'(x,x^*)&=~(\lambda+r)\cdot \dfrac{p(x,x^*)(p(x,x^*)-1)}{[(\lambda+r)-(\lambda+\mu)p(x,x^*)]\, x }\,.
\end{cases} 
\end{equation}
for all  $x\in [M,x^*]$ with $x^*\geq M$. Solving the above ODE (see in Section 5 of \cite{BMNP}), we obtain that 
\[
V(x,x^*)=B\cdot  \left(\dfrac{1-p(x,x^*)}{1-\theta(x^*)}\right)^{\frac{r}{r+\lambda}},\qquad p(x,x^*)=\dfrac{\theta(x^*) x^*}{x}\cdot \left(\dfrac{1-p(x,x^*)}{1-\theta(x^*)}\right)^{\frac{r-\mu}{r+\lambda}}
\]
for all $x\geq [M,x^*]$. Thus,
\[
\liminf_{x^*\to+\infty}V(x,x^*)\ge B\cdot \left(1-\frac{R}{x}\right)^{\frac{r}{r+\lambda}}\qquad\forall x\geq M
\]
and it yields \eqref{eq:bb}.
\par\medskip\par
\textbf{3.} We are now going to prove (ii). Assume that 
\begin{equation}\label{eq:as131}
\limsup_{s\to+\infty}~\theta(s)s=+\infty\,.
\end{equation}
Set 
\[\gamma\doteq \min\left\{\frac{r-\mu}{2c^{-1}(rB)}, \frac{(r-\mu)c'(0)}{4B}\right\}\qquad\textrm{ and }\qquad M_2\doteq \max\left\{\frac{4}{r-\mu},\frac{4B}{(r-\mu)L'(0)}\right\}\,.\]
For any $x^*>M_2$, denote by 
\[
\tau(x^*)\doteq \begin{cases}
x^*\qquad&\mathrm{if}\qquad \theta(x^*)\geq \gamma\,,\\[4mm]
\inf\left\{x\geq M_2~\Big|~p(x,x^*)\leq \gamma \right\} \qquad&\mathrm{if}\qquad\ \theta(x^*) < \gamma\,.
\end{cases} 
\]
From \eqref{eq:cod11}--\eqref{eq:v=0-D}, the decreasing property of $p$ yields
\begin{equation}\label{eq:upp}
p(x,x^*)\ge \gamma \qquad \forall x\in [M_2,\tau(x^*)[
\end{equation}
and
\[
p(x,x^*)<\gamma\qquad\Longrightarrow \qquad u(x,x^*)=v(x,x^*)\qquad \qquad\forall x\in [\tau(x^*),x^*]\,.
\]
As in the step 2, for any $x\in [\tau(x^*),x^*]$, we have 
\[
V(x,x^*)=B\cdot  \left(\dfrac{1-p(x,x^*)}{1-\theta(x^*)}\right)^{\frac{r}{r+\lambda}},\qquad p(x,x^*)=\dfrac{\theta(x^*) x^*}{x}\cdot \left(\dfrac{1-p(x,x^*)}{1-\theta(x^*)}\right)^{\frac{r-\mu}{r+\lambda}}
\]
This implies that 
\begin{equation}\label{eq:bv-e1}
V(x,x^*)=B\cdot \left(\frac{p(x,x^*)x}{\theta(x^*)x^*}\right)^{\frac{r}{r-\mu}}\le B\cdot \left(\frac{x}{\theta(x^*)x^*}\right)^{\frac{r}{r-\mu}}
\end{equation}
for all $x\in [\tau(x^*),x^*]$. On the other hand, for any $x\in [M_2,\tau(x^*)]$, from \eqref{eq:Sdode}, \eqref{eq:detHamil} and \eqref{eq:upp}, it holds 
\[
rV(x,x^*)\le \frac{r+\lambda}{p(x,x^*)}xV'(x,x^*)\le \frac{(r+\lambda)x}{\gamma}\cdot V'(x,x^*)\,.
\]
This implies that 
\begin{equation}\label{eq:bv-e2}
V(x,x^*)\le V(\tau(x^*),x^*)\cdot \left(\frac{x}{\tau(x^*)}\right)^{\frac{r\gamma}{r+\lambda}}\le B\cdot\left(\frac{x}{\tau(x^*)}\right)^{\frac{r\gamma}{r+\lambda}}
\end{equation}
for all $x\in [M_2,\tau(x^*)]$.
\par\medskip\par 
For any fix $x_0\geq M_2$, we will prove  that 
\begin{equation}\label{eq:pr}
\limsup_{x^*\to+\infty}~V(x_0,x^*)=0
\end{equation}
Two cases are considered:
\begin{itemize}
\item If $\limsup_{x^*\to+\infty}\tau(x^*)=+\infty$ then \eqref{eq:bv-e2} yields
\[
\lim_{x^*\to+\infty}V(x_0,x^*)\le \liminf_{x^*\to+\infty}~B\cdot \left(\frac{x_0}{\tau(x^*)}\right)^{\frac{r\gamma}{r+\lambda}}=0\,.
\]
\item If $\limsup_{x^*\to+\infty}\tau(x^*)<+\infty$ then 
\[
\tau(x^*)<M_3\qquad\forall x^*>0
\]
for some $M_3>0$. Recalling \eqref{eq:bv-e1} and \eqref{eq:as131}, we obtain that 
\[
\lim_{x^*\to\infty}~V(x_0,x^*)\le \lim_{x^*\to\infty}V(x_0+M_3,x^*)\le \lim_{x^*\to\infty}B\cdot \left(\frac{x_0+M_3}{\theta(x^*)x^*}\right)^{\frac{r}{r-\mu}}=0\,.
\]
\end{itemize}
Thus, \eqref{eq:pr} holds and the increasing property of $V(\cdot,x^*)$ yields \eqref{eq:vvD}.
\end{proof}

\par\medskip\par 

The following result shows that for sufficiently large initial DTI and bankruptcy threshold and recovery fraction after bankruptcy, the optimal strategy for the borrower will use currency devaluation to deflate the DTI. For simplicity, let us consider $x^*$ and $B^*$ sufficiently large such that

\begin{equation}\label{eq:cond12}
x^*>\dfrac{L'(0)+Br}{L'(0)\cdot (r-\mu)}\qquad\textrm{ and }\qquad B\geq\dfrac{2(r-\mu)c'(0)}{r}.
\end{equation}
In this case, the following holds:
\begin{proposition}[Devaluating strategies]
Let $x\mapsto \left(V(x,x^*),p(x,x^*)\right)$ be an equilibrium solution of \eqref{eq:Sdode} with boundary conditions \eqref{eq:bdc}. If 
\begin{equation}\label{eq:cond13}
\theta(x^*)x^*>\dfrac{2(r+\lambda)c'(0)}{r-\mu}\cdot\left(\dfrac{1}{rB}+\dfrac{1}{L'(0)}\right)
\end{equation}
then the function
\[v^*(x,x^*)=\underset{\omega\ge 0}{\mathrm{argmin}}\left\{c(\omega)-\omega x V'(x,x^*)\right\}\]
is not identically zero.
\end{proposition}
\begin{proof}
Set $M:=\dfrac{L'(0)+Br}{L'(0)\cdot (r-\mu)}$. Assume by a contradiction that $v^*(x,x^*)=0$ for all $x\in [M,x^*]$. In particular, we have 
\begin{equation}\label{eq:aspL}
0\le x V'(x,x^*)\le c'(0)\qquad x\in [M,x^*].
\end{equation}
The system \eqref{eq:Sdode} in $[M,x^*]$ reduces to 
\begin{equation}\label{eq:redSode}
\begin{cases}
rV(x)=\tilde{H}(x,V'(x),p(x))\cr\cr
(r+\lambda)(p(x)-1)=\tilde H_{\xi}(x,V'(x),p(x))\cdot p'(x)
\end{cases}\end{equation}
with
\[
\tilde{H}(x,\xi,p)~=~\min_{u\in [0,1]}\left\{ L(u) - \dfrac{u}{p}\, \xi \right\}+ \left(\dfrac{\lambda+r}{p} -\lambda-\mu\right) x\, \xi.
\]
Since  $r> \mu$ and $p\in [0,1]$, it holds
\begin{align*}
\tilde{H}(x,\xi,p)~\ge~-\dfrac{\xi}{p}+ \left(\lambda+r-p(\lambda+\mu)\right) x\, \dfrac{\xi}{p}
~\ge~\left((r-\mu)x-1\right)\cdot \dfrac{\xi}{p}
\end{align*}
and  \eqref{eq:redSode} yields  
\[
rB~\ge~rV(x,x^*)~\ge~\left((r-\mu)x-1\right)\cdot\dfrac{V'(x,x^*)}{p(x,x^*)}.
\]
Thus for $x\in \left[M,x^*\right]$ we obtain 
\[
\dfrac{V'(x,x^*)}{p(x,x^*)}~\le~ \dfrac{rB}{(r-\mu)x-1}\le L'(0),
\]
which immediately implies 
\[
u^*(x,x^*)~:=~\underset{u\in [0,1]}{\mathrm{argmin}}\left\{ L(u) - u\cdot\dfrac{V'(x,x^*)}{p(x,x^*)}\right\}~=~0.
\]
Hence, $(V(\cdot,x^*),p(\cdot,x^*))$ solves (\ref{eq:Sdode}) on $[M,x^*]$ and 
\[
V(x,x^*)=B\cdot\left(\dfrac{1-p(x,x^*)}{1-\theta(x^*)}\right)^{\frac{r}{r+\lambda}}\geq B\cdot  \left(1-\dfrac{r}{r+\lambda}\cdot p(x,x^*)\right)
\]
\[
p(x,x^*)=\dfrac{\theta(x^*) x^*}{x}\cdot \left(\dfrac{1-p(x,x^*)}{1-\theta(x^*)}\right)^{\frac{r-\mu}{r+\lambda}}\geq\dfrac{\theta(x^*) x^*}{x}\cdot  \left(1-\dfrac{r-\mu}{r+\lambda}\cdot p(x,x^*)\right)
\]
for all $x\in [M,x^*]$. From the above inequality, one derives 
\[
p(x,x^*)~\geq~\dfrac{(r+\lambda)\theta(x^*)x^*}{(r+\lambda )x+(r-\mu)\theta(x^*)x^*}.
\]
Thus, \eqref{eq:aspL} and the first equation in \eqref{eq:exp-odes} imply 
\begin{align*}
c'(0)\geq& xV'(x,x^*)~=~rp(x,x^*)\cdot \dfrac{V(x,x^*)}{(\lambda+r)-(\lambda+\mu)p(x,x^*)}\\
\geq&\dfrac{rp(x,x^*)B}{r+\lambda } \cdot \dfrac{r+\lambda -r p(x,x^*)}{(\lambda+r)-(\lambda+\mu)p(x,x^*)}\geq\dfrac{rp(x,x^*)B}{r+\lambda}\\
\geq&\dfrac{rB\theta(x^*)x^*}{(r+\lambda )x+(r-\mu)\theta(x^*)x^*}\qquad\forall x\in [M,x^*].
\end{align*}
In particular, choose $x=M$ and recall \eqref{eq:cond12}, we get 
\[
M\geq \dfrac{rB-(r-\mu)c'(0)}{(r+\lambda) c'(0)}\cdot {\theta(x^*)x^*}\geq\dfrac{rB}{2(r+\lambda)c'(0)}\cdot {\theta(x^*)x^*}
\]
and it contradicts to \eqref{eq:cond13}.
%
\end{proof}

\appendix

\section{Some results of convex analysis}\label{app:convex}

We introduce now some concepts of convex analysis, referring the reader to \cite{ET} and \cite{R} for a comprehensive introduction to the subject.

\begin{definition}[Convex conjugate and subdifferential]
We recall that the convex conjugate $F^\circ:\mathbb R^d\to\mathbb R\cup \{\pm\infty\}$ of a map $F:\mathbb R^d\to\mathbb R\cup\{+\infty\}$
is the lower semicontinuous convex function defined by 
\[F^\circ(z^*)=\sup_{z\in\mathbb R^d}\Big\{\langle z^*,z\rangle-F(z)\Big\}.\]
Let $F:\mathbb R^d\to\mathbb R\cup\{+\infty\}$ be proper (i.e., not identically $+\infty$), convex, lower semicontinuous functions, 
$x\in\mathrm{dom}\, F:=\{x\in\mathbb R^d:\,F(x)\in\mathbb R\}$. 
We define the \emph{subdifferential in the sense of convex analysis} of $F$ at $x$ by setting
\[\partial F(x):=\{v_x\in \mathbb R^d:\,F(y)-F(x)\ge \langle v_x,y-x\rangle\textrm{ for all }y\in \mathbb R^d\}.\]
\end{definition}

The following result provide a list of some properties of the subdifferential in the sense of convex analysis.

\begin{lemma}[Properties of the subdifferential]\label{lemma:subdifprop}
Let $F,G:\mathbb R^d\to\mathbb R\cup\{+\infty\}$ be proper (i.e., not identically $+\infty$), convex, lower semicontinuous functions, 
\begin{enumerate}
\item If $F$ is classically (Fr\'echet) differentiable at $x$, then $\partial F(x)=\{F'(x)\}$.
\item $z^*\in \partial F(z)$ if and only if $z\in \partial F^\circ(z^*)$.  
\item $F(x_0)=\displaystyle\min_{x\in \mathbb R^d} F(x)$ if and only if $0\in \partial F(x_0)$
\item $z^*\in \partial F^\circ(z)$ if and only if $F(z)+F^\circ(z^*)=\langle z^*,z\rangle$. In this case $z^*\in\mathrm{dom}\,F^\circ$;
\item $\lambda\ge 0$ we have $\partial(\lambda F)(z)=\lambda \partial F(z)$;
\item if there exists $z\in\mathrm{dom}(F)\cap \mathrm{dom}(G)$ such that $F$ is continuous at $z$ then
$\partial (F + G)(x) = \partial F (x) + \partial G(x)$ for all $x\in\mathrm{dom}(F)\cap \mathrm{dom}(G)$;
\item let $\bar y\in \mathbb R^m$, $\Lambda:\mathbb R^m\to \mathbb R^d$ be a linear map, $G$ be continuous and finite at $\Lambda(\bar y)$; 
Then $\partial(G\circ\Lambda)(y)=\Lambda^T\partial G(\Lambda y)$ for all $y\in \mathbb R^m$, where  $\Lambda^T:\mathbb R^d\to\mathbb R^m$
is the adjoint of $\Lambda$.
\end{enumerate}
\end{lemma}

\section{Properties of the Hamiltonian function}\label{app:Ham-prop}

In this Appendix we collect some technical results related to the Hamiltonian function in the stochastic and in the deterministic case.

\begin{lemma}\label{lemma:subdiffcost}
Assume \textbf{(A1)}-\textbf{(A2)}. Then, denoted by $L^\circ,c^\circ$ the convex conjugate (see Appendix \ref{app:convex} for the notation) of $L$ and $c$, respectively, 
we have that $L^\circ,c^\circ:\mathbb R\to\mathbb R$ are continuously differentiable and
\begin{align*}
(L^\circ)'(\rho)=\begin{cases}0,&\textrm{ if }\rho<L'(0),\\\\ (L')^{-1}(\rho),&\textrm{ if }\rho\ge L'(0),\end{cases}&&
(c^\circ)'(\rho)=\begin{cases}0,&\textrm{ if }\rho<c'(0),\\\\ (c')^{-1}(\rho),&\textrm{ if }\rho\ge c'(0),\end{cases}
\end{align*}
Moreover,
\begin{align*}L^\circ(\rho)\le\max\{0,\rho\},&&c^\circ(\rho)\le\max\{0,v_{\max}\rho\}\end{align*}
\end{lemma}
\begin{proof}
Recalling the assumptions $\textbf{(A1)}-\textbf{(A2)}$ on $L,c$, the equations
\begin{align*}
L^\circ(\rho_1)+L(u)=u\rho_1,&&c^\circ(\rho_2)+c(v)=v\rho_2,
\end{align*}
admits as unique solutions 
\begin{align*}
u(\rho_1)=\begin{cases}0,&\textrm{ if }\rho_1<L'(0),\\ \\(L')^{-1}(\rho_1),&\textrm{ if }\rho_1\ge L'(0)>0,\end{cases}&&
v(\rho_2)=\begin{cases}0,&\textrm{ if }\rho_2<c'(0),\\ \\(c')^{-1}(\rho_2),&\textrm{ if }\rho_2\ge c'(0)\ge 0.\end{cases}
\end{align*}
The result now follows from Theorem 23.5, Theorem 25.1, and Theorem 26.3 in \cite{R}.
For the second part, set $I_C(s)=0$ if $s\in C$ and $0$ otherwise, 
since $L(u)\ge I_{[0,1]}(u)$ and $c(v)\ge I_{[0,v_{\max}]}(v)$, we have
\begin{align*}
L^\circ(\rho)\le I^\circ_{[0,1]}(\rho)=&\max_{u\in [0,1]}\langle u,\rho\rangle=\max\{0,\rho\},\\
c^\circ(\rho)\le I^\circ_{[0,v_{\max}]}(\rho)=&\max_{v\in [0,v_{\max}]}\langle v,\rho\rangle=\max\{0,v_{\max}\cdot\rho\}.
\end{align*}
\end{proof}

\par\medskip\par

Lemma \ref{lemma:subdiffcost} immediately implies

\begin{lemma}\label{lemma:Hamgrad}
Assume \textbf{(A1)}-\textbf{(A2)}, and let $H$ be defined as in \eqref{eq:Hamiltonian}. Then $H$ is continuous differentiable and
its gradient at points $(x,\xi,p)\in [0,+\infty[\times[0,+\infty[\times]0,1]$ can be expressed in terms of 
$u^*(\xi,p):=(L^\circ)'(\xi/p)$ and $v^*(x,\xi):=(c^\circ)'(x\xi)$ by 
\begin{equation}\begin{cases}
\displaystyle H_{x}(x,\xi,p)=& \displaystyle \Big[(\lambda+r)-p(\lambda +\mu +v^*(x,\xi)-\sigma^2)\Big]\cdot \frac{\xi}{p},\\
\displaystyle H_{\xi}(x,\xi,p)=& \displaystyle \frac{1}{p}\cdot \Big[x\big((\lambda+r)-p(\lambda +\mu +v^*(x,\xi)-\sigma^2)\big)-u^*(\xi,p)\Big],\\
\displaystyle H_{p}(x,\xi,p)=& \displaystyle (u^*(\xi,p)-x(\lambda +r))\cdot \frac{\xi}{p^2},
\end{cases}
\end{equation}
Moreover, 
\[\begin{cases}
u^*(\xi,p)=&\underset{u\in [0,1]}{\mathrm{argmin}}\left\{ L(u) - u\,\dfrac{\xi}{p}\right\},\\ \\
v^*(x,\xi)=&\underset{v\ge 0}{\mathrm{argmin}}\left\{ c(v) - vx\xi\right\}.
\end{cases}\]
and for all $x>0$, $0<p\le 1$ we have also
\begin{align}
\label{eq:optuv}&\nabla u^*(\xi,p)=\dfrac{(1,-L'(u^*(x,\xi,p)))}{pL''(u_*(x,\xi,p))},\textrm{ if }\xi> pL'(0),\\
\nonumber &\nabla v^*(x,\xi)=\dfrac{(\xi,x)}{c''(v^*(x,\xi))},\textrm{ if }x\xi> c'(0),\\
\nonumber &\lim_{\xi\to +\infty}v^*(x,\xi)=v_{\max},
\end{align}
\end{lemma}

\begin{lemma}\label{lemma:Hprop}
Let the assumptions \textbf{(A1)}-{\bf(A2)} hold. 
Then 
\begin{enumerate}
\item for all $\xi\geq 0$ and $p\in ]0,1]$,  the function $H$ in \eqref{eq:Hamiltonian} satisfies
\begin{align*}
H(x,\xi,p)\le&\left(\dfrac{\lambda+r}{p} -(\lambda+\mu) + \sigma^2 \right) x\xi;\\
H(x,\xi,p)\ge&\left(\frac{(\lambda+r) x-1}{p}+(\sigma^2-\lambda-\mu-v^*(x,\xi))x\right)\cdot \xi\\
\ge&\left(\frac{(\lambda+r) x-1}{p}+(\sigma^2-(\lambda+\mu)-v_{\max})x\right)\cdot \xi;\\
H_\xi(x,\xi,p)\le&\left(\dfrac{\lambda+r}{p} -(\lambda+\mu) + \sigma^2\right) x;\\
H_\xi(x,\xi,p)\ge&\frac{(\lambda+r) x-1}{p}+(\sigma^2-(\lambda+\mu)-v^*(x,\xi))x\\
\ge&\frac{(\lambda+r)x-1}{p}+(\sigma^2-(\lambda+\mu)-v_{\max})x
\end{align*}
\item for every $x,p>0$ 
the map $\xi\mapsto H(x,\xi,p)$ is concave down and satisfies
\begin{align*}
H(x,0,p)=0,&&H_\xi(x,0,p)=\left(\dfrac{\lambda+r}{p} -(\lambda+\mu)+\sigma^2\right) x.
\end{align*}
\end{enumerate}
\end{lemma}
\begin{proof}
The concavity of $\xi\mapsto H(x,\xi,p)$ for every $x,p>0$ is immediate from the definition of $H$ in \eqref{eq:Hamiltonian}. 
The equalities in item $(2)$ are immediate from Lemma \ref{lemma:subdiffcost}.
The upper bound on $H(x,\xi,p)$ follows from the positivity of $L^\circ$ and $c^\circ$.
By concavity, the map $\xi\mapsto H_\xi(x,\xi,p)$ is monotone decreasing, thus $H_\xi(x,\xi,p)\le H_\xi(x,0,p)$, which proves the upper bound
on $H_\xi(x,\xi,p)$ together with item (2). The lower estimate for $H(x,\xi,p)$ comes from the second part of Lemma \ref{lemma:subdiffcost},
in particular from the upper estimate on $L^\circ(\cdot)$.
The lower estimate for $H_\xi(x,\xi,p)$ comes from Lemma \ref{lemma:Hamgrad}, noticing that
\[\lim_{\xi\to +\infty}u^*(\xi,p)=\lim_{\rho\to +\infty}(L')^{-1}(\rho)=1,\lim_{\xi\to +\infty}v^*(x,\xi)=\lim_{\rho\to +\infty}(c')^{-1}(\rho)=v_{\max},\]
and using the decreasing property of $\xi\mapsto H_\xi(x,\xi,p)$, i.e., the fact that
\[\lim_{\zeta\to +\infty}H_\xi(x,\zeta,p)\le H_\xi(x,\xi,p),\]
for all $x\ge 0$, $p\in ]0,1]$, $\xi\in\mathbb R$.
\end{proof}

\begin{lemma}\label{lemma:holder-gen}
Assume that $f:I\to\mathbb R$ is a $C^2$ convex strictly increasing function defined on a real interval $I$, and satisfying $f''\ge \delta >0$.
Then, denoted by $g$ its inverse function, $g:f(I)\to I$, we have that $g$ is $1/2$-H\"older continuous.
\end{lemma}
\begin{proof}
Indeed, let $x_1,x_2\in f^{-1}(I)$ with $x_1\le x_2$, and set $y_1=g(x_1)$ and $y_2=g(x_2)$.
\begin{align*}
f(y_2)-f(y_1)~&=~\int_{y_1}^{y_2} f'(t)\,dt=\int_{y_1}^{y_2} [f'(t)-f'(y_1)]\,dt\\ &=~f'(y_1) \cdot (y_2-y_1)+\int_{y_1}^{y_2}\int_{y_1}^t f''(s)\,ds\,dt\\ 
&\ge~f'(y_1) \cdot (y_2-y_1)+\dfrac{\delta }{2} (y_2-y_1)^2\ge \dfrac{\delta }{2} (y_2-y_1)^2,
\end{align*}
since $f$ is strictly increasing, $f''(s)\ge \delta $, and $y_1\le y_2$.
Thus if $x_2\ge x_1$ we have
\[|g(x_2)-g(x_1)|\le \sqrt{\dfrac{2}{\delta }}|x_2-x_1|^{1/2}.\]
By switching the roles of $x_2$ and $x_1$, the same holds true if $x_1\ge x_2$.
\end{proof}

\section{List of symbols}\label{app:list}

\begin{tabular}{cl}
$X=X(t)$&the total nominal value of the outstanding debt;\\
$Y=Y(t)$&the gross national product GDP measured in terms\\ &of the floating currency unit;\\
$U=U(t)$&the rate of payments that the borrower chooses\\ &to make to the lenders;\\
$\mu$   &average growth rate of the economy;\\
$\sigma$&the volatility;\\
$\lambda$&rate at which the borrower pays back the principal;\\
$r$      &discount rate;\\
$x=x(t)=X(t)/Y(t)$&debt-to-GDP ratio (DTI);\\
$p=p(t)$&discounted bond price;\\
$\tilde u=\tilde u(t)=U(t)/Y(t)$&fraction of the total income allocated to reduce the debt;\\
$\tilde v=\tilde v(t)$&devaluation rate;\\
$x^*$   &maximum DTI threshold;\\
$\theta$&recovery fraction after bankruptcy;\\
$B$     &bankruptcy cost;\\
$L=L(u)$&adversion toward austerity policy;\\
$c=c(v)$&social cost due to currency devaluation.
\end{tabular}

\end{document}